\documentclass{amsart}%
\usepackage{eurosym}
\usepackage{amsmath}
\usepackage{amssymb}
\usepackage{amsfonts}
\usepackage{graphicx}%
\newtheorem{theorem}{Theorem}[section]
\theoremstyle{plain}

\newtheorem{corollary}[theorem]{Corollary}

\newtheorem{definition}{Definition}

\newtheorem{lemma}[theorem]{Lemma}

\newtheorem{proposition}[theorem]{Proposition}

\numberwithin{equation}{section}
\begin{document}
\title[Solutions to an elliptic equation concentrating at a submanifold]{Solutions to a singularly perturbed supercritical elliptic equation on a
Riemannian manifold concentrating at a submanifold}
\author{M\'{o}nica Clapp}
\address{Instituto de Matem\'{a}ticas, Universidad Nacional Aut\'{o}noma de M\'{e}xico,
Circuito Exterior CU, 04510 M\'{e}xico DF, Mexico}
\email{monica.clapp@im.unam.mx}
\author{Marco Ghimenti}
\address{Dipartimento di Matematica Applicata, Universit\`{a} di Pisa, Via Buonarroti
1/c 56127, Pisa, Italy}
\email{marco.ghimenti@dma.unipi.it}
\author{Anna Maria Micheletti}
\address{Dipartimento di Matematica Applicata, Universit\`{a} di Pisa, Via Buonarroti
1/c 56127, Pisa, Italy}
\email{a.micheletti@dma.unipi.it}
\thanks{Research supported by CONACYT grant 129847 and UNAM-DGAPA-PAPIIT grant
IN106612 (Mexico), and MIUR projects PRIN2009: 
``Variational and Topological Methods in the Study of Nonlinear Phenomena" and 
``Critical Point Theory and Perturbative Methods for Nonlinear Differential Equations". (Italy).}
\date{\today}
\maketitle

\begin{abstract}
Given a smooth Riemannian manifold $(\mathcal{M},g)$ we investigate the
existence of positive solutions to the equation
\[
-\varepsilon^{2}\Delta_{g}u+u=u^{p-1}\ \text{ on }\mathcal{M}%
\]
which concentrate at some submanifold of $\mathcal{M}$ as $\varepsilon
\rightarrow0,$ for supercritical nonlinearities. We obtain a posive answer for
some manifolds, which include warped products.

Using one of the projections of the warped product or some harmonic morphism,
we reduce this problem to a problem of the form
\[
-\varepsilon^{2}\text{div}_{h}\left(  c(x)\nabla_{h}u\right)
+a(x)u=b(x)u^{p-1},
\]
with the same exponent $p,$ on a Riemannian manifold $(M,h)$ of smaller
dimension, so that $p$ turns out to be subcritical for this last problem.
Then, applying Lyapunov-Schmidt reduction, we establish existence of a
solution to the last problem which concentrates at a point as $\varepsilon
\rightarrow0.$

\end{abstract}

\section{Introduction}

Let $(\mathfrak{M},\mathfrak{g})$ be a compact smooth Riemannian manifold,
without boundary, of dimension $m\geq2$. We consider the problem
\[
(\wp_{\varepsilon})\qquad\left\{
\begin{array}
[c]{l}%
-\varepsilon^{2}\Delta_{\mathfrak{g}}v+v=v^{p-1},\\
v\in H_{\mathfrak{g}}^{1}(\mathfrak{M}),\text{\quad}v>0.
\end{array}
\right.
\]
where $p>2$ and $\varepsilon^{2}$ is a singular perturbation parameter. The
space $H_{\mathfrak{g}}^{1}(\mathfrak{M})$ is the completion of $\mathcal{C}%
^{\infty}(\mathfrak{M})$ with respect to the norm defined by $\Vert
v\Vert_{\mathfrak{g}}^{2}:=\int_{\mathfrak{M}}(\left\vert \nabla
_{\mathfrak{g}}v\right\vert ^{2}+v^{2})d\mu_{\mathfrak{g}}$.

Let $2_{m}^{\ast}:=\infty$ if $m=2$ and $2_{m}^{\ast}:=\frac{2m}{m-2}$ if
$m\geq3$ be the critical Sobolev exponent in dimension $m.$ In the subcritical
case, where $p<2_{m}^{\ast},$ solutions to $(\wp_{\varepsilon})$ which
concentrate at a point are known to exist. In \cite{bp} Byeon and Park showed
that there are solutions with one peak concentrating at a maximum point of the
scalar curvature of $(\mathfrak{M},\mathfrak{g})$ as $\varepsilon\rightarrow
0$. Single-peak solutions concentrating at a stable critical point of the
scalar curvature of $(\mathfrak{M},\mathfrak{g})$ as $\varepsilon\rightarrow0$
were obtained by Micheletti and Pistoia in \cite{mp}, whereas in \cite{dmp}
Dancer, Micheletti and Pistoia proved the existence solutions with $k$ peaks
which concentrate at an isolated minimum of the scalar curvature of
$(\mathfrak{M},\mathfrak{g})$ as $\varepsilon\rightarrow0.$ Some results on
sign changing solutions, as well as multiplicity results, are also available,
see \cite{mp2, bbm, cy, v, h, gm}.

We are specially interested in the critical and the supercritical case, where
$p\geq2_{m}^{\ast},$ and in solutions exhibiting concentration at positive
dimensional submanifolds of $\mathfrak{M}$ as $\varepsilon\rightarrow0$.

For the analogous equation%
\begin{equation}
-\varepsilon^{2}\Delta v+V(x)v=v^{p-1}\quad\text{in }\Omega,\label{omega}%
\end{equation}
in a bounded smooth domain $\Omega$ in $\mathbb{R}^{n}$ with Dirichlet or
Neumann boundary conditions, or in the entire space $\mathbb{R}^{n}$, there is
a vast literature concerning solutions concentrating at one or at a finite set
of points. Various results concerning concentration at an $(n-1)$-dimensional
sphere are nowadays also available, see e.g. \cite{am, amn1, amn2, bd, be-d,
mm}\ and the references therein.

A fruitful approach to produce solutions to equation (\ref{omega}) which
concentrate at other positive dimensional manifolds, for $p$ up to some
supercritical exponent, is to reduce it to an equation in a domain of lower
dimension. This approach was introduced by Ruf and Srikanth in \cite{rs},
where the reduction is given by a Hopf map. Reductions may also be performed
by means of other maps which preserve the laplacian, or by considering
rotational symmetries, or by a combination of both, as has been recently done
in \cite{acp, cfp, cfp2, kp, kp2, pp, ps, wy} for different problems.

Next, we describe some of these reductions in the context of Riemannian
manifolds. For simplicity, from now on we consider all manifolds to be smooth,
compact and without boundary.

\subsection{Harmonic morphisms}

\label{subsec:hm}Let $(\mathfrak{M},\mathfrak{g})$ and $(M,g)$ be Riemannian
manifolds of dimensions $m$ and $n$ respectively. A \emph{harmonic morphism}
is a horizontally conformal submersion $\pi:\mathfrak{M}\rightarrow M$ with
dilation $\lambda:\mathfrak{M}\rightarrow\lbrack0,\infty)$ which satisfies the
equation%
\begin{equation}
(n-2)\mathcal{H}(\nabla_{\mathfrak{g}}\ln\lambda)+(m-n)\kappa^{\mathcal{V}}=0,
\label{eq:hm}%
\end{equation}
where $\kappa^{\mathcal{V}}$ is the mean curvature of the fibers and
$\mathcal{H}$ is the projection of the tangent space of $\mathfrak{M}$ onto
the space orthogonal to the fibers of $\pi$, see \cite{bw, er, f, w}. Typical
examples are the Hopf fibrations
\[
\mathbb{S}^{m}\rightarrow\mathbb{R}P^{m},\text{\quad}\mathbb{S}^{2m+1}%
\rightarrow\mathbb{C}P^{m},\text{\quad}\mathbb{S}^{4m+3}\rightarrow
\mathbb{H}P^{m},\text{\quad}\mathbb{S}^{15}\rightarrow\mathbb{S}^{8},
\]
which are Riemannian submersions (i.e. $\lambda\equiv1$) with totally geodesic
fibers $\mathbb{S}^{0}$, $\mathbb{S}^{1},$ $\mathbb{S}^{3}$and $\mathbb{S}%
^{7}$ respectively, see \cite[Examples 2.4.14-2.4.17]{bw}. So they trivially
satisfy (\ref{eq:hm}).

The main property of a harmonic morphism is that it preserves the
Laplace-Beltrami operator, i.e. it satisfies%
\[
\Delta_{\mathfrak{g}}(u\circ\pi)=\lambda^{2}\left[  (\Delta_{g}u)\circ
\pi\right]
\]
for every $\mathcal{C}^{2}$-function $u:M\rightarrow\mathbb{R}.$ The following
proposition is an immediate consequence of this fact.

\begin{proposition}
\label{prop:hm}Let $\pi:\mathfrak{M}\rightarrow M$ be a harmonic morphism with
dilation $\lambda$. Assume there exists $\mu:M\rightarrow(0,\infty)$ such that
$\mu\circ\pi=\lambda^{2}.$ Then $u:M\rightarrow\mathbb{R}$ solves%
\begin{equation}
-\varepsilon^{2}\Delta_{g}u+\frac{1}{\mu(x)}u=\frac{1}{\mu(x)}u^{p-1}%
\text{\qquad on }M \label{eq:M}%
\end{equation}
iff $v:=u\circ\pi:\mathfrak{M}\rightarrow\mathbb{R}$ solves
\begin{equation}
-\varepsilon^{2}\Delta_{\mathfrak{g}}v+v=v^{p-1}\text{\qquad on }\mathfrak{M}.
\label{eq:Mfrac}%
\end{equation}

\end{proposition}

Note that the exponent $p$ is the same in both equations. If $n<m$ then
$2_{m}^{\ast}<2_{n}^{\ast},$ so $p\in\lbrack2_{m}^{\ast},2_{n}^{\ast})$ is
subcritical for problem (\ref{eq:M}) but it is critical or supercritical for
problem (\ref{eq:Mfrac}). Moreover, if solutions $u_{\varepsilon}$ of
(\ref{eq:M})\ concetrate at some point $x_{0}\in M$ the corresponding
solutions $v_{\varepsilon}:=u_{\varepsilon}\circ\pi$ to problem
(\ref{eq:Mfrac})\ concentrate at $\pi^{-1}(x_{0}),$ which is a manifold of
dimension $m-n.$

\subsection{Warped products}

\label{sec:wp}A different type of reduction can be performed on warped
products. If $(M,g)$ and $(N,h)$ are Riemannian manifolds of dimensions $n$
and $k$ respectively, and $f:M\rightarrow(0,\infty)$ is a smooth function, the
\emph{warped product} $M\times_{f^{2}}N$ is the cartesian product $M\times N$
equipped with the Riemannian metric $g+f^{2}h.$

The projection $\pi_{N}:M\times_{f^{2}}N\rightarrow N$ is a harmonic morphism
with dilation $\lambda(y,z)=\frac{1}{f(y)}$ \cite[Example 2.4.26]{bw} but
there is no function $\mu:M\rightarrow(0,\infty)$ such that $\mu\circ
\pi=\lambda^{2}$ unless $f$ is constant. So we cannot apply Proposition
\ref{prop:hm}. Instead, we consider the projection $\pi_{M}:M\times_{f^{2}%
}N\rightarrow M.$ A straightforward computation gives the following result,
cf. \cite{dl}.

\begin{proposition}
\label{prop:wp}$u:M\rightarrow\mathbb{R}$ solves%
\begin{equation}
-\varepsilon^{2}\text{\emph{div}}_{g}\left(  f^{k}(x)\nabla_{g}u\right)
+f^{k}(x)u=f^{k}(x)u^{p-1}\text{\qquad on }M \label{eq:div}%
\end{equation}
iff $v:=u\circ\pi_{M}:M\times_{f^{2}}N\rightarrow\mathbb{R}$ solves
\begin{equation}
-\varepsilon^{2}\Delta_{\mathfrak{g}}v+v=v^{p-1}\text{\qquad on }%
M\times_{f^{2}}N. \label{eq:wp}%
\end{equation}

\end{proposition}

The exponent $p$ is again the same in both equations. So, if $k>0,$ then
$p\in\lbrack2_{n+k}^{\ast},2_{n}^{\ast})$ is subcritical for problem
(\ref{eq:div}) but it is critical or supercritical for problem (\ref{eq:wp}).
Moreover, if solutions $u_{\varepsilon}$ of (\ref{eq:div})\ concetrate at some
point $x_{0}\in M,$ then the solutions $v_{\varepsilon}:=u_{\varepsilon}%
\circ\pi_{M}$ to problem (\ref{eq:wp})\ concentrate at $\pi_{M}^{-1}%
(x_{0})\cong(N,f^{2}(x_{0})h).$

\subsection{The reduced problem}

Propositions \ref{prop:hm}\ and \ref{prop:wp}\ lead to study the problem%
\begin{equation}
\left\{
\begin{array}
[c]{l}%
-\varepsilon^{2}\text{div}_{g}\left(  c(x)\nabla_{g}u\right)
+a(x)u=b(x)u^{p-1}\\
u\in H_{g}^{1}(M),\text{\quad}u>0.
\end{array}
\right.  \label{eq:main}%
\end{equation}
where $(M,g)$ is an $n$-dimensional compact smooth Riemannian manifold without
boundary, $n\geq2$, $p\in(2,2_{n}^{\ast})$, $a,b,c$ are positive real-valued
$\mathcal{C}^{2}$-functions on $M$, and $\varepsilon$ is a positive parameter.

In order to state our main result we need the following definition.

\begin{definition}
\label{def:stable}Let $\phi\in\mathcal{C}^{1}(M,\mathbb{R})$. A set
$K\subset\left\{  x\in M:\nabla_{g}\phi(x)=0\right\}  $ is a $\mathcal{C}^{1}%
$\emph{-stable critical set for} $\phi$ if for each $\mu>0$ there exists
$\delta>0$ such that, if $\psi\in\mathcal{C}^{1}(M,\mathbb{R})$ with
\[
\max_{d_{g}(x,K)\leq\mu}\left\vert \phi(x)-\psi(x)\right\vert +\left\vert
\nabla_{g}\phi(x)-\nabla_{g}\psi(x)\right\vert \leq\delta,
\]
then $\psi$ has a critical point $x_{0}$ with $d_{g}(x_{0},K)\leq\mu$. Here
$d_{g}$ denotes the geodesic distance associated to the Riemannian metric $g$.
\end{definition}

We shall prove the following result.

\begin{theorem}
\label{thm:main}Let $K$ be a $\mathcal{C}^{1}$-stable critical set for the
function
\[
\Gamma(x):=\frac{c(x)^{\frac{n}{2}}a(x)^{\frac{p}{p-2}-\frac{n}{2}}%
}{b(x)^{\frac{2}{p-2}}}.
\]
Then, for every $\varepsilon$ small enough, problem \emph{(\ref{eq:main})} has
a solution $u_{\varepsilon}$ which concentrates at a point $x_{0}\in K$ as
$\varepsilon\rightarrow0$.
\end{theorem}

This result, together with Propositions \ref{prop:hm}\ and \ref{prop:wp},
implies the existence of solutions to problem $(\wp_{\varepsilon})$ which
concentrate at a submanifold for subcritical, critical and supercritical
exponents. More precisely, the following results hold true.

\begin{corollary}
\label{cor:hm}Assume there exists a harmonic morphism $\pi:\mathfrak{M}%
\rightarrow M$ whose dilation $\lambda:\mathfrak{M}\rightarrow(0,\infty)$ is
of the form $\lambda^{2}=\mu\circ\pi.$ Let $n:=\dim M\geq2$ and let $K$ be a
$\mathcal{C}^{1}$-stable critical set for the function $\mu^{\frac{n-2}{2}%
}:M\rightarrow(0,\infty).$ Then, for every $p\in(2,2_{n}^{\ast})$ and
$\varepsilon$ small enough, problem \emph{(\ref{eq:main})} has a solution
$v_{\varepsilon}$ which concentrates at the fiber $\pi^{-1}(x_{0})$ over some
point $x_{0}\in K,$ as $\varepsilon\rightarrow0.$
\end{corollary}

\begin{corollary}
\label{cor:wp}Let $\mathfrak{M}$\ be the warped product $M\times_{f^{2}}N,$
$n:=\dim M\geq2,$ $k:=\dim N$ and $K$ be a $\mathcal{C}^{1}$-stable critical
set for the function $f^{k}:M\rightarrow(0,\infty).$ Then for every
$p\in(2,2_{n}^{\ast})$ and $\varepsilon$ small enough, problem
\emph{(\ref{eq:main})} has a solution $v_{\varepsilon}$ which concentrates at
$\{x_{0}\}\times N,$ for some point $x_{0}\in K,$ as $\varepsilon\rightarrow0$.
\end{corollary}

\subsection{Surfaces of revolution}

The following application illustrates how one can combine these results to
produce solutions which concentrate at different submanifolds.

\begin{theorem}
\label{thm:appl}Let $M$ be a compact smooth Riemannian submanifold of
$\mathbb{R}^{\ell}\times\left(  0,\infty\right)  ,$ without boundary, $n:=\dim
M\geq1,$ and let%
\[
\mathfrak{M}:=\{(y,z)\in\mathbb{R}^{\ell}\times\mathbb{R}^{k+1}:(y,\left\vert
z\right\vert )\in M\},
\]
with the Riemannian metric inherited from $\mathbb{R}^{\ell+k+1}.$ The
following statements hold true:

\begin{enumerate}
\item[(a)] For each $p\in(2,2_{n+k}^{\ast})$ problem $(\wp_{\varepsilon}%
)$\ has a solution $v_{\varepsilon}^{0}$ which concentrates at two points in
$\mathfrak{M}$ as $\varepsilon\rightarrow0$.

\item[(b)] If $k$ is odd then, for each $p\in(2,2_{n+k}^{\ast})$ and each
$j\in\mathbb{N},$ problem $(\wp_{\varepsilon})$\ has a solution
$v_{\varepsilon}^{0,j}$ which concentrates at $j$ points in $\mathfrak{M}$ as
$\varepsilon\rightarrow0$.

\item[(c)] If $k\geq3$ is odd then, for each $p\in(2,2_{n+k-1}^{\ast}),$
problem $(\wp_{\varepsilon})$\ has a solution $v_{\varepsilon}^{1}$ which
concentrates at a $1$-dimensional sphere in $\mathfrak{M}$ as $\varepsilon
\rightarrow0$.

\item[(d)] If $k=4m+3$ then, for each $p\in(2,2_{n+k-3}^{\ast}),$ problem
$(\wp_{\varepsilon})$\ has a solution $v_{\varepsilon}^{3}$ which concentrates
at a $3$-dimensional sphere in $\mathfrak{M}$ as $\varepsilon\rightarrow0.$

\item[(e)] If $k=15$ then, for each $p\in(2,2_{n+8}^{\ast}),$ problem
$(\wp_{\varepsilon})$\ has a solution $v_{\varepsilon}^{7}$ which concentrates
at a $7$-dimensional sphere in $\mathfrak{M}$ as $\varepsilon\rightarrow0.$

\item[(f)] If $n\geq2$ then, for each $p\in(2,2_{n}^{\ast}),$ problem
$(\wp_{\varepsilon})$\ has a solution $v_{\varepsilon}^{k}$ which concentrates
at a $k$-dimensional sphere in $\mathfrak{M}$ as $\varepsilon\rightarrow0$.
\end{enumerate}
\end{theorem}

\begin{proof}
Observe that $\mathfrak{M}$ is isometric to the warped product $M\times
_{f^{2}}\mathbb{S}^{k}=(M\times\mathbb{S}^{k},g+f^{2}h)$ where $g$ is the
euclidean metric on $M$, $h$ is the standard metric on $\mathbb{S}^{k}$ and
$f:M\rightarrow(0,\infty)$ is the projection $f(y,t):=t.$ The isometry is
given by%
\[
\mathfrak{M}\ni(y,z)\longmapsto\left(  (y,\left\vert z\right\vert ),\frac
{z}{\left\vert z\right\vert }\right)  \in M\times_{f^{2}}\mathbb{S}^{k}.
\]

(a): \ The product $\pi:M\times_{f^{2}}\mathbb{S}^{k}\rightarrow
M\times_{f^{2}}\mathbb{R}P^{k}$ of the identity map on $M$ with the Hopf
fibration is a Riemannian submersion whose fiber is diffeomorphic to
$\mathbb{S}^{0}.$ Hence, it is a harmonic morphism with dilation
$\lambda\equiv1.$ Theorem 1.2 in \cite{mp} yields a positive solution
$u_{\varepsilon}^{0}\in H_{g}^{1}(M\times_{f^{2}}\mathbb{R}P^{k})$ to equation
(\ref{eq:M}) with $\mu\equiv1,$ which concentrates at a $\mathcal{C}^{1}%
$-stable critical point $\xi_{0}$ of the scalar curvature of $M\times_{f^{2}%
}\mathbb{R}P^{k},$ for every $p\in(2,2_{n+k}^{\ast}).$ By Proposition
\ref{prop:hm}, $v_{\varepsilon}:=u_{\varepsilon}\circ\pi$ is a solution to
$(\wp_{\varepsilon})$ which concentrates at the pair of points $\pi^{-1}%
(\xi_{0})$ as $\varepsilon\rightarrow0$.

(b): \ If $k$ is odd we identify $\mathbb{R}^{k+1}$ with $\mathbb{C}%
^{\frac{k+1}{2}}.$ The group $\Gamma_{j}:=\{\mathrm{e}^{2\pi\mathrm{i}%
s/j}:s=0,...,j-1\}$ acts by multiplication on each coordinate of
$\mathbb{C}^{\frac{k+1}{2}}.$ Since this action is free on $\mathbb{S}^{k},$
the orbit space $\mathbb{S}^{k}/\Gamma_{j}$ is a smooth manifold. We endow
$\mathbb{S}^{k}/\Gamma_{j}$ with the Riemannian metric which turns the orbit
map $\mathbb{S}^{k}\rightarrow\mathbb{S}^{k}/\Gamma_{j}$ into a Riemannian
submersion. Now we can argue as in (a), replacing the Hopf map by this orbit map.

(c): \ Let $k=2m+1$ and $\mathfrak{N}:=(M\times\mathbb{C}P^{m},f^{\frac
{2}{d-2}}g+f^{\frac{2}{d-2}+2}\widetilde{h})$ where $d:=2m+n=\dim
(M\times\mathbb{C}P^{m})$ and $\widetilde{h}$ is the standard metric on
$\mathbb{C}P^{m}.$ If we identify $\mathfrak{M}$\ with $(M\times
\mathbb{S}^{2m+1},g+f^{2}h)$, the product of the identity map on $M$ with the
Hopf fibration $\mathbb{S}^{2m+1}\rightarrow\mathbb{C}P^{m}$ is a horizontally
conformal submersion $\pi:\mathfrak{M}\rightarrow\mathfrak{N}$ with dilation
$\lambda(y,z)=f(y,\left\vert z\right\vert )^{\frac{1}{d-2}}=\left\vert
z\right\vert ^{\frac{1}{d-2}},$ whose fiber is diffeomorphic to $\mathbb{S}%
^{1}.$ The mean curvature of the fibers is $\kappa^{\mathcal{V}}%
(y,z)=-\frac{z}{\left\vert z\right\vert ^{2}}.$ Therefore, $\lambda$ satisfies
equation (\ref{eq:hm}) and, consequently, $\pi$ it is a harmonic morphism.
Moreover, $\lambda^{2}=\mu\circ\pi,$ where $\mu((y,t),\zeta)=t^{\frac{2}{d-2}%
}$ for $(y,t)\in M,$ $\zeta\in\mathbb{C}P^{m}$. Corollary \ref{cor:hm} implies
that, for every $p\in(2,2_{d}^{\ast})$ and $\varepsilon$ small enough, problem
(\ref{eq:main}) has a solution $v_{\varepsilon}^{1}$ which concentrates at
some fiber of $\pi$ over a $\mathcal{C}^{1}$-stable critical set for the
function $((y,t),\zeta)\mapsto t$ as $\varepsilon\rightarrow0.$

(d): \ Let $\mathfrak{N}:=(M\times\mathbb{H}P^{m},f^{\frac{6}{d-2}}%
g+f^{\frac{6}{d-2}+2}\widetilde{h})$ where $d:=4m+n=\dim(M\times
\mathbb{H}P^{m})$ and $\widetilde{h}$ is the standard metric on $\mathbb{H}%
P^{m}.$ If we identify $\mathfrak{M}$\ with $(M\times\mathbb{S}^{4m+3}%
,g+f^{2}h)$, the product of the identity map on $M$ with the Hopf fibration
$\mathbb{S}^{4m+3}\rightarrow\mathbb{H}P^{m}$ is a horizontally conformal
submersion $\pi:\mathfrak{M}\rightarrow\mathfrak{N}$ with dilation
$\lambda(y,z)=\left\vert z\right\vert ^{\frac{3}{d-2}},$ whose fiber is
diffeomorphic to $\mathbb{S}^{3}.$ The mean curvature of the fibers is
$\kappa^{\mathcal{V}}(y,z)=-\frac{z}{\left\vert z\right\vert ^{2}}.$
Therefore, $\lambda$ satisfies equation (\ref{eq:hm}) and $\pi$ it is a
harmonic morphism. Corollary \ref{cor:hm} implies that, for every
$p\in(2,2_{d}^{\ast})$ and $\varepsilon$ small enough, problem (\ref{eq:main})
has a solution $v_{\varepsilon}^{3}$ which concentrates at some fiber of $\pi$
over a $\mathcal{C}^{1}$-stable critical set for the function $((y,t),\zeta
)\mapsto t^{3}$ as $\varepsilon\rightarrow0.$

(e): \ The proof is similar to that of (c) and (d), this time using the Hopf
fibration $\mathbb{S}^{15}\rightarrow\mathbb{S}^{8}.$

(f): \ This follows immediately from Corollary \ref{cor:wp}.
\end{proof}

Note that the proof contains information on the location of the sets of
concentration. We have recently learned that a similar statement for annuli in
$\mathbb{R}^{N}$ was proved by Ruf and Srikanth \cite{RSp}.

The rest of this paper is devoted to the proof of Theorem \ref{thm:main},
which is based on the well-known Lyapunov-Schmidt reduction. The outline of
the paper is as follows: In Section \ref{sec:limprob} we discuss the limit
problem. In Section \ref{sec:proof} we outline the Lyapunov-Schmidt procedure
and use it to prove Theorem \ref{thm:main}. In Section \ref{sec:finitedim} we
establish the finite dimensional reduction, and in Section \ref{sec:redfunc}%
\ we obtain the expansion of the reduced functional. We collect some technical
facts in Appendix \ref{sec:appendix}.

\section{The limit problem}

\label{sec:limprob}Let $(M,g)$ be an $n$-dimensional compact smooth Riemannian
manifold without boundary, $n\geq2,$ $\ a,b,c$ be positive real-valued
$\mathcal{C}^{2}$-functions on $M$, and $\varepsilon$ be a positive parameter.

We denote by $H_{\varepsilon}$ the Sobolev space $H_{g}^{1}(M)$ with the
scalar product
\[
\left\langle u,v\right\rangle _{\varepsilon}:=\frac{1}{\varepsilon^{n-2}}%
\int_{M}c(x)\nabla_{g}u\nabla_{g}vd\mu_{g}+\frac{1}{\varepsilon^{n}}\int
_{M}a(x)uvd\mu_{g}%
\]
and norm
\[
\Vert u\Vert_{\varepsilon}:=\left\langle u,u\right\rangle _{\varepsilon}%
^{1/2}=\frac{1}{\varepsilon^{n-2}}\int_{M}c(x)|\nabla_{g}u|^{2}d\mu_{g}%
+\frac{1}{\varepsilon^{n}}\int_{M}a(x)u^{2}d\mu_{g}.
\]
Similarly, we denote by $L_{\varepsilon}^{q}$ be the Lebesgue space $L_{g}%
^{q}(M)$ endowed with the norm
\[
|u|_{q,\varepsilon}:=\left(  \frac{1}{\varepsilon^{n}}\int_{M}|u|^{q}d\mu
_{g}\right)  ^{1/q}.
\]
We recall that $|u|_{q,\varepsilon}\leq C\Vert u\Vert_{\varepsilon}$ for
$q\in\lbrack2,2_{n}^{\ast}]$, where the constant $C$ does not depend on
$\varepsilon$.

Fix $p\in(2,2_{n}^{\ast})$ and set%
\[
{A(x):=\frac{a(x)}{c(x)},}\text{\qquad}{B(x):=\frac{b(x)}{c(x)},}\text{\qquad
}\gamma(\xi):=\left(  \frac{a(\xi)}{b(\xi)}\right)  ^{\frac{1}{p-2}}.
\]
For $\xi_{0}\in M$ let $V=V^{\xi_{0}}$ be the unique positive spherically
symmetric solution to
\begin{equation}
-c(\xi_{0})\Delta V+a(\xi_{0})V=b(\xi_{0})V^{p-1}\text{\quad in }%
\mathbb{R}^{n}.\label{eq:Vxi-1}%
\end{equation}
This is the limit equation for problem (\ref{eq:main}) in the tangent space
$T_{\xi_{0}}M\equiv\mathbb{R}^{n}.$ It is equivalent to
\begin{equation}
-\Delta V+A(\xi_{0})V=B(\xi_{0})V^{p-1}\text{\quad in }\mathbb{R}%
^{n}.\label{eq:Vxi}%
\end{equation}
A simple computation shows that
\begin{equation}
V^{\xi_{0}}(z)=\gamma(\xi_{0})U(\sqrt{A(\xi_{0})}z),\label{eq:rescaling}%
\end{equation}
where $U$ is the unique positive spherically symmetric solution of
\begin{equation}
-\Delta U+U=U^{p-1}\text{\quad in }\mathbb{R}^{n}.\label{eq:U}%
\end{equation}

Fix $r>0$ smaller than the injectivity radius of $M.$ Then, the exponential
map $\exp_{\xi}:B(0,r)\rightarrow B_{g}(\xi,r)$ is a diffeomorphism for every
$\xi\in M$. Here the tangent space $T_{\xi}M$ is identified with
$\mathbb{R}^{n}$, $B(0,r)$ is the ball of radius $r$ in $\mathbb{R}^{n}$
centered at $0,$ and $B_{g}(\xi,r)$ denotes the ball of radius $r$ in $M$
centered at $\xi$ with respect to the distance induced by the Riemannian
metric $g$. Let $\chi\in\mathcal{C}^{\infty}(\mathbb{R}^{n})$ be a radial
cut-off function such that $\chi(z)=1$ if $\left\vert z\right\vert \leq r/2$
and $\chi(z)=0$ if $\left\vert z\right\vert \geq r.$ For $\xi\in M$ and
$\varepsilon>0$ we define $W_{\varepsilon,\xi}\in H_{g}^{1}(M)$ by
\[
W_{\varepsilon,\xi}(x):=\left\{
\begin{array}
[c]{ll}%
V^{\xi}\left(  \frac{1}{\varepsilon}\exp_{\xi}^{-1}(x)\right)  \chi\left(
\exp_{\xi}^{-1}(x)\right)  & \text{\qquad if }x\in B_{g}(\xi,r),\\
0 & \text{\qquad otherwise.}%
\end{array}
\right.
\]
Setting $V_{\varepsilon}(z):=V\left(  \frac{z}{\varepsilon}\right)  $ and
$y:=\exp_{\xi}^{-1}x$ we have that
\[
W_{\varepsilon,\xi}(\exp_{\xi}(y))=V^{\xi}\left(  \frac{y}{\varepsilon
}\right)  \chi(y)=V_{\varepsilon}^{\xi}(y)\chi(y),
\]
so the function $W_{\varepsilon,\xi}$ is simply the function $V^{\xi}$
rescaled, cut off and read in a normal neighborhood of $\xi$ in $M.$

Similarly, for $i=1,\dots,n$ we define
\[
Z_{\varepsilon,\xi}^{i}(x):=\left\{
\begin{array}
[c]{ll}%
\psi_{\xi}^{i}\left(  \frac{1}{\varepsilon}\exp_{\xi}^{-1}(x)\right)
\chi\left(  \exp_{\xi}^{-1}(x)\right)  & \text{\qquad if }x\in B_{g}(\xi,r),\\
0 & \text{\qquad otherwise,}%
\end{array}
\right.
\]
where
\[
\psi_{\xi}^{i}(\eta)=\frac{\partial}{\partial\eta_{i}}V^{\xi}(\eta)=\gamma
(\xi)\sqrt{A(\xi)}\frac{\partial U}{\partial\eta_{i}}(\sqrt{A(\xi)}\eta).
\]
The functions $\psi_{\xi}^{i}$ are solutions to the linearized equation
\begin{equation}
-\Delta\psi+A(\xi)\psi=(p-1)B(\xi)\left(  V^{\xi}\right)  ^{p-2}%
\psi\text{\quad in }\mathbb{R}^{n}. \label{eq:lin}%
\end{equation}

Next, we compute the derivatives of $W_{\varepsilon,\xi}$ with respect to
$\xi$ in a normal neighborhood. Fix $\xi_{0}\in M$. We write the points
$\xi\in B_{g}(\xi_{0},r)$ as
\[
\xi=\xi(y):=\exp_{\xi_{0}}(y)\text{\qquad with }y\in B(0,r),
\]
and consider the function
\[
\mathcal{E}(y,x):=\exp_{\xi(y)}^{-1}(x)=\exp_{\exp_{\xi_{0}}(y)}^{-1}(x)
\]
defined on the set $\{(y,x):y\in B(0,r),$ $x\in B_{g}(\xi(y),r)\}$. Then we
can write
\begin{align*}
W_{\varepsilon,\xi(y)}(x)  &  =\gamma(\xi(y))U_{\varepsilon}(\sqrt{A(\xi
(y))}\exp_{\xi(y)}^{-1}(x))\chi(\exp_{\xi(y)}^{-1}(x))\\
&  =\tilde{\gamma}(y)U_{\varepsilon}(\sqrt{\tilde{A}(y)}\mathcal{E}%
(y,x))\chi(\mathcal{E}(y,x))
\end{align*}
where $\tilde{A}(y):=A(\exp_{\xi_{0}}(y))$ and $\tilde{\gamma}(y):=\gamma
(\exp_{\xi_{0}}(y))$. Thus, we have
\begin{align}
\frac{\partial}{\partial y_{1}}W_{\varepsilon,\xi(y)}  &  =\left(
\frac{\partial}{\partial y_{1}}\tilde{\gamma}(y)\right)  U_{\varepsilon}%
(\sqrt{\tilde{A}(y)}\mathcal{E}(y,x))\chi(\mathcal{E}(y,x))\nonumber\\
&  +\tilde{\gamma}(y)U_{\varepsilon}(\sqrt{\tilde{A}(y)}\mathcal{E}%
(y,x))\frac{\partial\chi}{\partial z_{k}}(\mathcal{E}(y,x))\frac{\partial
}{\partial y_{1}}\mathcal{E}_{k}(y,x)\label{eq:derWeps}\\
&  +\frac{1}{\varepsilon}\tilde{\gamma}(y)\chi(\mathcal{E}(y,x))\frac
{\partial}{\partial z_{k}}\left(  U_{\varepsilon}(\sqrt{\tilde{A}%
(y)}\mathcal{E}(y,x))\right)  \frac{\partial}{\partial y_{1}}\mathcal{E}%
_{k}(y,x).\nonumber
\end{align}

One has the Taylor expansions
\begin{equation}
\frac{\partial}{\partial y_{h}}\mathcal{E}_{k}(0,\exp_{\xi_{0}}\varepsilon
z)=-\delta_{hk}+O(\varepsilon^{2}|z|^{2}), \label{eq:espE}%
\end{equation}%
\begin{align}
g^{ij}(\varepsilon z)  &  =\delta_{ij}+\frac{\varepsilon^{2}}{2}\sum
_{r,k=1}^{n}\frac{\partial^{2}g^{ij}}{\partial z_{r}\partial z_{k}}%
(0)z_{r}z_{k}+O(\varepsilon^{3}|z|^{3})=\delta_{ij}+o(\varepsilon
),\label{eq:espg1}\\
\left\vert g(\varepsilon z)\right\vert ^{\frac{1}{2}}  &  =1-\frac
{\varepsilon^{2}}{4}\sum_{i,r,k=1}^{n}\frac{\partial^{2}g^{ii}}{\partial
z_{r}\partial z_{k}}(0)z_{r}z_{k}+O(\varepsilon^{3}|z|^{3})=1+o(\varepsilon),
\label{eq:espg2}%
\end{align}
where, as usual,
\[
(g^{ij}(z))\text{ is the inverse matrix of }(g_{ij}(z))\text{\qquad and\qquad
}\left\vert g(z)\right\vert :=\det(g_{ij}(z)).
\]

\begin{proposition}
\label{prop:Zi}There exists a positive constant $C$ such that
\[
\left(  Z_{\varepsilon,\xi}^{h},Z_{\varepsilon,\xi}^{k}\right)  _{\varepsilon
}=C\delta_{hk}+o(1).
\]

\end{proposition}

\begin{proof}
Using the Taylor expansions of $g^{ij}(\varepsilon z)$, $|g(\varepsilon
z)|^{\frac{1}{2}}$, $a(\exp_{\xi}(\varepsilon z))$ and $c(\exp_{\xi
}(\varepsilon z))$ we obtain
\begin{align*}
\left\langle Z_{\varepsilon,\xi}^{h},Z_{\varepsilon,\xi}^{k}\right\rangle
_{\varepsilon}  &  =\frac{1}{\varepsilon^{n}}\int_{M}\varepsilon^{2}%
c(x)\nabla_{g}Z_{\varepsilon,\xi}^{h}(x)\nabla_{g}Z_{\varepsilon,\xi}%
^{k}(x)+a(x)Z_{\varepsilon,\xi}^{h}(x)Z_{\varepsilon,\xi}^{k}(x)d\mu_{g}\\
=  &  \int_{B(0,r/\varepsilon)}\sum_{ij}c(\exp_{\xi}(\varepsilon z))g_{\xi
}^{ij}(\varepsilon z)\frac{\partial}{\partial z_{i}}(\psi_{\xi}^{h}%
(z)\chi(\varepsilon z))\frac{\partial}{\partial z_{j}}(\psi_{\xi}^{h}%
(z)\chi(\varepsilon z))|g_{\xi}(\varepsilon z)|^{\frac{1}{2}}dz\\
&  +\int_{B(0,r/\varepsilon)}a(\exp_{\xi}(\varepsilon z))\psi_{\xi}^{h}%
(z)\psi_{\xi}^{h}(z)\chi^{2}(\varepsilon z)dz\\
=  &  c(\xi)\int_{\mathbb{R}^{n}}\nabla\psi_{\xi}^{h}\nabla\psi_{\xi}%
^{h}dz+a(\xi)\int_{\mathbb{R}^{n}}\psi_{\xi}^{h}\psi_{\xi}^{k}dz+o(1)\\
=  &  C\delta_{hk}+o(1),
\end{align*}
as claimed.
\end{proof}

\section{Outline of the proof of Theorem \ref{thm:main}}

\label{sec:proof}Let%
\[
K_{\varepsilon,\xi}:=\text{span}\left\{  Z_{\varepsilon,\xi}^{1}%
,\dots,Z_{\varepsilon,\xi}^{n}\right\}
\]
and
\[
K_{\varepsilon,\xi}^{\bot}:=\left\{  \phi\in H_{\varepsilon}:\left\langle
\phi,Z_{\varepsilon,\xi}^{i}\right\rangle _{\varepsilon}=0,\ i=1,\dots
,n\right\}
\]
be its orthogonal complement in $H_{\varepsilon}.$\ We denote the orthogonal
projections onto these subspaces by
\[
\Pi_{\varepsilon,\xi}:H_{\varepsilon}\rightarrow K_{\varepsilon,\xi
}\text{\qquad and\qquad}\Pi_{\varepsilon,\xi}^{\bot}:H_{\varepsilon
}\rightarrow K_{\varepsilon,\xi}^{\bot}.
\]

Let $i_{\varepsilon}^{\ast}:L_{\varepsilon}^{p^{\prime}}\rightarrow
H_{\varepsilon}$ be the adjoint operator of the Sobolev embedding
$i_{\varepsilon}:H_{\varepsilon}\rightarrow L_{\varepsilon}^{p}$. It is well
known that
\begin{align}
\Vert i_{\varepsilon}^{\ast}(v)\Vert_{\varepsilon}  &  \leq C_{1}%
|v|_{p^{\prime},\varepsilon}\qquad\forall v\in L_{g}^{p^{\prime}%
},\label{eq:istar}\\
|u|_{p,\varepsilon}  &  \leq C_{2}\Vert u\Vert_{\varepsilon}\qquad\forall u\in
H_{\varepsilon}, \label{eq:istella}%
\end{align}
where the constants $C_{1},C_{2}$ do not depend on $\varepsilon$.

We look for a solution to problem (\ref{eq:main}) of the form%
\[
u_{\varepsilon}=W_{\varepsilon,\xi}+\phi\text{\quad with }\phi\in
K_{\varepsilon,\xi}^{\bot}.
\]
This is equivalent to solving the pair of equations
\begin{align}
\Pi_{\varepsilon,\xi}^{\bot}\left\{  W_{\varepsilon,\xi}+\phi-i_{\varepsilon
}^{\ast}\left(  b(x)f(W_{\varepsilon,\xi}+\phi)\right)  \right\}   &
=0,\label{eq:red1}\\
\Pi_{\varepsilon,\xi}\left\{  W_{\varepsilon,\xi}+\phi-i_{\varepsilon}^{\ast
}\left(  b(x)f(W_{\varepsilon,\xi}+\phi)\right)  \right\}   &  =0,
\label{eq:red2}%
\end{align}
where, to simplify notation, we have set $f(u):=(u^{+})^{p-1}$.

Next we state the results needed to prove our main result.

\begin{proposition}
\label{prop:phieps}There exist $\varepsilon_{0}>0$ and $C>0$ such that, for
each $\xi\in M$ and each $\varepsilon\in(0,\varepsilon_{0}),$ there exists a
unique $\phi_{\varepsilon,\xi}\in K_{\varepsilon,\xi}^{\bot}$ which solves
\emph{(\ref{eq:red1})}. It satisfies
\[
\Vert\phi_{\varepsilon,\xi}\Vert_{\varepsilon}<C\varepsilon.
\]
Moreover, $\xi\mapsto\phi_{\varepsilon,\xi}$ is a $\mathcal{C}^{1}$-map.
\end{proposition}

\begin{proof}
The proof will be given in Section \ref{sec:finitedim}.
\end{proof}

A solution to problem (\ref{eq:main}) is a critical point of the energy
functional $J_{\varepsilon}:H_{\varepsilon}\rightarrow\mathbb{R}$ given by
\[
J_{\varepsilon}(u)=\frac{1}{2}\Vert u\Vert_{\varepsilon}^{2}-\frac
{1}{p\varepsilon^{n}}\int_{M}b(x)(u^{+})^{p}d\mu_{g}.
\]
Proposition \ref{prop:phieps} allows us to define the reduced energy
functional $\tilde{J}_{\varepsilon}:M\rightarrow\mathbb{R}$ as
\[
\tilde{J}_{\varepsilon}(\xi):=J_{\varepsilon}(W_{\varepsilon,\xi}%
+\phi_{\varepsilon,\xi}).
\]
It has the following property.

\begin{proposition}
\label{lem:tool1}If $\xi_{0}\in M$ is a critical point of $\tilde
{J}_{\varepsilon}$, then the function $W_{\varepsilon,\xi_{0}}+\phi
_{\varepsilon,\xi_{0}}$ is a solution of problem \emph{(\ref{eq:main})}.
\end{proposition}

\begin{proof}
Set $\xi=\xi(y)=\exp_{\xi_{0}}(y)$. If $\xi_{0}$ is a critical point for
$\tilde{J}_{\varepsilon}$ we have
\[
\left.  \frac{\partial}{\partial y_{h}}\tilde{J}_{\varepsilon}(\exp_{\xi_{0}%
}(y))\right\vert _{y=0}=0\text{\qquad for all }h=1,\dots,n.
\]
Since $\phi_{\varepsilon,\xi(y)}$ solves (\ref{eq:red1}), we get
\begin{multline*}
\frac{\partial}{\partial y_{h}}\tilde{J}_{\varepsilon}(\exp_{\xi_{0}%
}(y))=J_{\varepsilon}^{\prime}(W_{\varepsilon,\xi(y)}+\phi_{\varepsilon
,\xi(y)})\left[  \frac{\partial}{\partial y_{h}}\left(  W_{\varepsilon,\xi
(y)}+\phi_{\varepsilon,\xi(y)}\right)  \right]  \\
=\left\langle W_{\varepsilon,\xi(y)}+\phi_{\varepsilon,\xi(y)}-i_{\varepsilon
}^{\ast}\left(  b(x)f(W_{\varepsilon,\xi(y)}+\phi_{\varepsilon,\xi
(y)})\right)  ,\frac{\partial}{\partial y_{h}}\left(  W_{\varepsilon,\xi
(y)}+\phi_{\varepsilon,\xi(y)}\right)  \right\rangle _{\varepsilon}\\
=\sum_{l=1}^{n}C_{\varepsilon}^{l}\left\langle Z_{\varepsilon,\xi(y)}%
,\frac{\partial}{\partial y_{h}}W_{\varepsilon,\xi(y)}+\frac{\partial
}{\partial y_{h}}\phi_{\varepsilon,\xi(y)}\right\rangle _{\varepsilon}%
\end{multline*}
for some $C_{\varepsilon}^{1},\dots,C_{\varepsilon}^{n}\in\mathbb{R}$. We want
to prove that, for $y=0$, $C_{\varepsilon}^{1}=\dots=C_{\varepsilon}^{n}=0$.
Since $\phi_{\varepsilon,\xi(y)}\in K_{\varepsilon,\xi(y)}^{\bot}$ we have%
\begin{align*}
\left\langle Z_{\varepsilon,\xi(y)},\left.  \frac{\partial}{\partial y_{h}%
}\phi_{\varepsilon,\xi(y)}\right\vert _{y=0}\right\rangle _{\varepsilon}= &
-\left\langle \left.  \frac{\partial}{\partial y_{h}}Z_{\varepsilon,\xi
(y)}\right\vert _{y=0},\phi_{\varepsilon,\xi(y)}\right\rangle _{\varepsilon}\\
= &  O\left(  \left\Vert \left.  \frac{\partial}{\partial y_{h}}%
Z_{\varepsilon,\xi(y)}\right\vert _{y=0}\right\Vert _{\varepsilon}%
\cdot\left\Vert \phi_{\varepsilon,\xi(y)}\right\Vert _{\varepsilon}\right)
=O(1)
\end{align*}
by Lemma \ref{lem:6.1} and Proposition \ref{prop:phieps}. Moreover, by Lemma
\ref{lem:6.2} we have
\[
\left\langle Z_{\varepsilon,\xi_{0}}^{l},\left.  \frac{\partial}{\partial
y_{h}}W_{\varepsilon,\xi(y)}\right\vert _{y=0}\right\rangle _{\varepsilon
}=-\frac{1}{\varepsilon}C\delta_{hl}+o\left(  \frac{1}{\varepsilon}\right)  ,
\]
thus
\[
0=\sum_{l=1}^{n}C_{\varepsilon}^{l}\left\langle Z_{\varepsilon,\xi(y)},\left.
\frac{\partial}{\partial y_{h}}\phi_{\varepsilon,\xi(y)}\right\vert
_{y=0}+\left.  \frac{\partial}{\partial y_{h}}W_{\varepsilon,\xi
(y)}\right\vert _{y=0}\right\rangle _{\varepsilon}=-\frac{C}{\varepsilon}%
\sum_{l=1}^{n}C_{\varepsilon}^{l}\left(  \delta_{hl}+o\left(  1\right)
\right)
\]
and this implies that $C_{\varepsilon}^{l}=0$ for all $l=1,\dots,n$.
\end{proof}

\begin{proposition}
\label{lem:tool2}The reduced energy is given by
\[
\tilde{J}_{\varepsilon}(\xi)=\left(  \frac{p-2}{2p}\int_{\mathbb{R}^{n}}%
U^{p}dz\right)  \frac{c(\xi)^{\frac{n}{2}}a(\xi)^{\frac{p}{p-2}-\frac{n}{2}}%
}{b(\xi)^{\frac{2}{p-2}}}+O(\varepsilon),
\]
$\mathcal{C}^{1}$-uniformly with respect to $\xi$ as $\varepsilon\rightarrow0$.
\end{proposition}

\begin{proof}
The proof follows from two lemmas which we prove in Section \ref{sec:redfunc}%
:\ Lemma \ref{lem:expJeps} asserts that $\tilde{J}_{\varepsilon}%
(\xi)=J_{\varepsilon}(W_{\varepsilon,\xi})+O(\varepsilon)$ as $\varepsilon
\rightarrow0,$ $\mathcal{C}^{1}$-uniformly with respect to $\xi$. Then, in
Lemma \ref{lem:redc0}, we obtain the expansion of $J_{\varepsilon
}(W_{\varepsilon,\xi})$ which yields the claim.
\end{proof}

Using the previous propositions we prove our main result.\medskip

\begin{proof}
[Proof of Theorem \ref{thm:main}.]If $K$ is a $\mathcal{C}^{1}$-stable
critical set for $\Gamma(\xi)=\frac{c(\xi)^{\frac{n}{2}}a(\xi)^{\frac{p}%
{p-2}-\frac{n}{2}}}{b(\xi)^{\frac{2}{p-2}}}$ then, by Definition
\ref{def:stable} and Proposition \ref{lem:tool2}, $\tilde{J}_{\varepsilon}$
has a critical point $\xi_{\varepsilon}\in M$ such that $d_{g}(\xi
_{\varepsilon},K)\rightarrow0$ as $\varepsilon\rightarrow0$. Proposition
\ref{lem:tool1} asserts that $u_{\varepsilon}:=W_{\varepsilon,\xi
_{\varepsilon}}+\Phi_{\varepsilon,\xi_{\varepsilon}}$ is a solution of
(\ref{eq:main}).
\end{proof}

\section{The finite dimensional reduction}

\label{sec:finitedim}In this section we solve equation (\ref{eq:red1}). We
introduce the linear operator
\begin{align*}
L_{\varepsilon,\xi}  &  :K_{\varepsilon,\xi}^{\bot}\rightarrow K_{\varepsilon
,\xi}^{\bot}\\
L_{\varepsilon,\xi}(\phi)  &  :=\Pi_{\varepsilon,\xi}^{\bot}\left\{
\phi-i_{\varepsilon}^{\ast}\left[  b(\cdot)f^{\prime}(W_{\varepsilon,\xi}%
)\phi\right]  \right\}  .
\end{align*}
Equation (\ref{eq:red1}) can be rewritten as
\[
L_{\varepsilon,\xi}(\phi)=N_{\varepsilon,\xi}(\phi)+R_{\varepsilon,\xi},
\]
where $N_{\varepsilon,\xi}$ is the nonlinear term
\[
N_{\varepsilon,\xi}(\phi):=\Pi_{\varepsilon,\xi}^{\bot}\left\{  i_{\varepsilon
}^{\ast}\left[  b(\cdot)\left(  f(W_{\varepsilon,\xi}+\phi)-f(W_{\varepsilon
,\xi})-f^{\prime}(W_{\varepsilon,\xi})\right)  \phi\right]  \right\}
\]
and $R_{\varepsilon,\xi}$ is the remainder
\[
R_{\varepsilon,\xi}:=\Pi_{\varepsilon,\xi}^{\bot}\left\{  i_{\varepsilon
}^{\ast}\left[  b(\cdot)f(W_{\varepsilon,\xi})\right]  -W_{\varepsilon,\xi
}\right\}  .
\]

\begin{lemma}
\label{lem:Linv}There exist $\varepsilon_{0}$ and $C>0$ such that, for every
$\xi\in M$ and $\varepsilon\in(0,\varepsilon_{0}),$
\[
\Vert L_{\varepsilon,\xi}(\phi)\Vert_{\varepsilon}\geq C\Vert\phi
\Vert_{\varepsilon}\text{\qquad for all \ }\phi\in K_{\varepsilon,\xi}^{\bot
}.
\]

\end{lemma}

\begin{proof}
Arguing by contradiction, assume there exist sequences $\varepsilon
_{k}\rightarrow0$, $\xi_{k}\in M$ with $\xi_{k}\rightarrow\xi\in M,$ and
$\phi_{k}\in K_{\varepsilon_{k},\xi_{k}}^{\bot}$ with $\Vert\phi_{k}%
\Vert_{\varepsilon_{k}}=1,$ such that
\[
L_{\varepsilon_{k},\xi_{k}}(\phi_{k})=:\psi_{k}\text{\quad satisfies\quad
}\Vert\psi_{k}\Vert_{\varepsilon_{k}}\rightarrow0\text{ as }k\rightarrow
+\infty.
\]
Let $\zeta_{k}\in K_{\varepsilon_{k},\xi_{k}}$ be such that
\begin{equation}
\phi_{k}-i_{\varepsilon_{k}}^{\ast}\left[  b(\cdot)f^{\prime}(W_{\varepsilon
_{k},\xi_{k}})\phi_{k}\right]  =\psi_{k}+\zeta_{k}. \label{eq:zetagreca}%
\end{equation}

Next, we prove that $\Vert\zeta_{k}\Vert_{\varepsilon_{k}}\rightarrow0$ as
$k\rightarrow+\infty$. Let ${\zeta_{k}=%
%TCIMACRO{\tsum _{j=1}^{n}}%
%BeginExpansion
{\textstyle\sum_{j=1}^{n}}
%EndExpansion
}${$\alpha$}${_{j}^{k}Z_{\varepsilon_{k},\xi_{k}}^{j}}$. Multiplying
(\ref{eq:zetagreca}) by $Z_{\varepsilon_{k},\xi_{k}}^{h}$ and noting that
$\phi_{k},\psi_{k}\in K_{\varepsilon_{k},\xi_{k}}^{\bot}$ we obtain
\begin{align}
\sum_{j=1}^{n}\alpha_{j}^{k}\left\langle Z_{\varepsilon_{k},\xi_{k}}%
^{j},Z_{\varepsilon_{k},\xi_{k}}^{h}\right\rangle _{\varepsilon_{k}}  &
=-\left\langle i_{\varepsilon_{k}}^{\ast}\left[  b(\cdot)f^{\prime
}(W_{\varepsilon_{k},\xi_{k}})\phi_{k}\right]  ,Z_{\varepsilon_{k},\xi_{k}%
}^{h}\right\rangle _{\varepsilon_{k}}\nonumber\\
&  =-\frac{1}{\varepsilon_{k}^{n}}\int_{M}b(x)f^{\prime}(W_{\varepsilon
_{k},\xi_{k}})\phi_{k}Z_{\varepsilon_{k},\xi_{k}}^{h}d\mu_{g}. \label{eq:aj}%
\end{align}
Set
\[
\tilde{\phi}_{k}:=\left\{
\begin{array}
[c]{ll}%
\phi_{k}\left(  \exp_{\xi_{k}}(\varepsilon_{k}z)\right)  \chi(\varepsilon
_{k}z) & \text{\quad if }z\in B(0,r/\varepsilon_{k}),\\
0 & \text{\quad otherwise.}%
\end{array}
\right.
\]
It is easy to prove that $\Vert\tilde{\phi}_{k}\Vert_{H^{1}(\mathbb{R}^{n}%
)}\leq C\Vert\phi_{k}\Vert_{\varepsilon_{k}}\leq C$ for some positive constant
$C$. Thus, there exists $\tilde{\phi}\in H^{1}(\mathbb{R}^{n})$ such that, up
to a subsequence, $\tilde{\phi}_{k}\rightarrow\tilde{\phi}$ weakly in
$H^{1}(\mathbb{R}^{n})$ and strongly in $L_{\text{loc}}^{p}(\mathbb{R}^{n})$
for all $p\in(2,2_{n}^{\ast})$. Since $\phi_{k}\in K_{\varepsilon_{k},\xi_{k}%
}^{\bot}$ we get
\begin{align*}
\sum_{j=1}^{n}\alpha_{j}^{k}  &  \left\langle Z_{\varepsilon_{k},\xi_{k}}%
^{j},Z_{\varepsilon_{k},\xi_{k}}^{h}\right\rangle _{\varepsilon_{k}}=-\frac
{1}{\varepsilon_{k}^{n}}\int_{M}b(x)f^{\prime}(W_{\varepsilon_{k},\xi_{k}%
})\phi_{k}Z_{\varepsilon_{k},\xi_{k}}^{h}d\mu_{g}\\
&  =\left\langle \phi_{k},Z_{\varepsilon_{k},\xi_{k}}^{h}\right\rangle
_{\varepsilon}-\frac{1}{\varepsilon_{k}^{n}}\int_{M}b(x)f^{\prime
}(W_{\varepsilon_{k},\xi_{k}})\phi_{k}Z_{\varepsilon_{k},\xi_{k}}^{h}d\mu
_{g}\\
&  =\frac{1}{\varepsilon_{k}^{n}}\int_{M}\left[  \varepsilon_{k}^{2}%
c(x)\nabla\phi_{k}\nabla Z_{\varepsilon_{k},\xi_{k}}^{h}+a(x)Z_{\varepsilon
_{k},\xi_{k}}^{h}\phi_{k}-b(x)f^{\prime}(W_{\varepsilon_{k},\xi_{k}})\phi
_{k}Z_{\varepsilon_{k},\xi_{k}}^{h}\right]  d\mu_{g}\\
&  =\int_{B(0,r/\varepsilon_{k})}\sum_{l,m=1}^{n}g^{lm}c(\exp_{\xi_{k}%
}(\varepsilon_{k}z))(\varepsilon_{k}z)\frac{\partial\tilde{\phi}_{k}}{\partial
z_{l}}(z)\frac{\partial\psi_{\xi_{k}}^{h}}{\partial z_{m}}(z)|g(\varepsilon
_{k}z)|^{\frac{1}{2}}dz\\
&  \qquad+\int_{B(0,r/\varepsilon_{k})}a(\exp_{\xi_{k}}(\varepsilon
_{k}z))\tilde{\phi}_{k}(z)\psi_{\xi_{k}}^{h}(z)|g(\varepsilon_{k}z)|^{\frac
{1}{2}}dz\\
&  \qquad-\int_{B(0,r/\varepsilon_{k})}b(\exp_{\xi_{k}}(\varepsilon
_{k}z))f^{\prime}(V^{\xi_{k}}(z)\chi(\varepsilon_{k}z))\tilde{\phi}_{k}%
(z)\psi_{\xi_{k}}^{h}(z)|g(\varepsilon_{k}z)|^{\frac{1}{2}}dz\\
&  =\int_{\mathbb{R}^{n}}c(\xi_{k})\nabla\tilde{\phi_{k}}\nabla\psi_{\xi_{k}%
}^{h}+a(\xi_{k})\tilde{\phi}_{k}\psi_{\xi_{k}}^{h}-b(\xi_{k})f^{\prime}%
(V^{\xi_{k}})\tilde{\phi}_{k}\psi_{\xi_{k}}^{h}dz+o(1)=o(1),
\end{align*}
because $\psi_{\xi_{k}}^{h}\rightarrow\psi_{\xi}^{h}$, $\ V^{\xi_{k}%
}\rightarrow V^{\xi}$ strongly in $H^{1}(\mathbb{R}^{n})$ and $\psi_{\xi}^{h}$
is a weak solution of the linearized equation (\ref{eq:lin}). By Proposition
\ref{prop:Zi} we have that $\left\langle Z_{\varepsilon_{k},\xi_{k}}%
^{j},Z_{\varepsilon_{k},\xi_{k}}^{h}\right\rangle _{\varepsilon_{k}}%
=C\delta_{jh}+o(1)$, where $C>0$. So we conclude that
\[
C\alpha_{h}^{k}+o(1)=\sum_{j=1}^{n}\alpha_{j}^{k}\left\langle Z_{\varepsilon
_{k},\xi_{k}}^{j},Z_{\varepsilon_{k},\xi_{k}}^{h}\right\rangle _{\varepsilon
_{k}}=o(1).
\]
This implies that $\alpha_{h}^{k}\rightarrow0$ for each $h=1,\dots,n$, and,
consequently, that $\Vert\zeta_{k}\Vert_{\varepsilon_{k}}\rightarrow0$.

Setting $u_{k}:=\phi_{k}-\psi_{k}-\zeta_{k}$, equation (\ref{eq:zetagreca})
can be read as
\begin{equation}
-\varepsilon_{k}^{2}\text{div}_{g}\left(  c(x)\nabla_{g}u_{k}\right)
+a(x)u_{k}=b(x)f^{\prime}(W_{\varepsilon_{k},\xi_{k}})u_{k}+b(x)f^{\prime
}(W_{\varepsilon_{k},\xi_{k}})(\psi_{k}+\zeta_{k}), \label{eq:ukappa}%
\end{equation}
where $\Vert u_{k}\Vert_{\varepsilon_{k}}\rightarrow1.$ Multiplying
(\ref{eq:ukappa}) by $u_{k}$ and integrating by parts we get
\begin{equation}
\Vert u_{k}\Vert_{\varepsilon_{k}}^{2}=\frac{1}{\varepsilon_{k}^{n}}\int
_{M}b(x)\left(  f^{\prime}(W_{\varepsilon_{k},\xi_{k}})u_{k}^{2}+f^{\prime
}(W_{\varepsilon_{k},\xi_{k}})(\psi_{k}+\zeta_{k})u_{k}\right)  .
\label{eq:u1}%
\end{equation}
From Holder's inequality, recalling that $|u|_{\varepsilon,p}\leq C\Vert
u\Vert_{\varepsilon}$, we obtain
\begin{align}
\frac{1}{\varepsilon_{k}^{n}}\int_{M}b(x)  &  f^{\prime}(W_{\varepsilon
_{k},\xi_{k}})(\psi_{k}+\zeta_{k})u_{k}\nonumber\\
&  \leq C\left(  \frac{1}{\varepsilon_{k}^{n}}\int_{M}f^{\prime}%
(W_{\varepsilon_{k},\xi_{k}})^{\frac{n}{2}}\right)  ^{\frac{2}{n}}%
|u_{k}|_{\varepsilon_{k},\frac{2n}{n-2}}^{\frac{n-2}{2n}}|\psi_{k}+\zeta
_{k}|_{\varepsilon_{k},\frac{2n}{n-2}}^{\frac{n-2}{2n}}\nonumber\\
&  \leq C\left(  \frac{1}{\varepsilon_{k}^{n}}\int_{M}f^{\prime}%
(W_{\varepsilon_{k},\xi_{k}})^{\frac{n}{2}}\right)  ^{\frac{2}{n}}\Vert
u_{k}\Vert_{\varepsilon_{k}}^{\frac{n-2}{2n}}\Vert\psi_{k}+\zeta_{k}%
\Vert_{\varepsilon_{k}}^{\frac{n-2}{2n}}\nonumber\\
&  \leq C\Vert u_{k}\Vert_{\varepsilon_{k}}^{\frac{n-2}{2n}}\Vert\psi
_{k}+\zeta_{k}\Vert_{\varepsilon_{k}}^{\frac{n-2}{2n}}=o(1), \label{eq:u2}%
\end{align}
because
\begin{align}
\frac{1}{\varepsilon_{k}^{n}}\int_{M}f^{\prime}(W_{\varepsilon_{k},\xi_{k}%
})^{\frac{n}{2}}d\mu_{g}  &  \leq\frac{1}{\varepsilon_{k}^{n}}\int_{B_{g}%
(\xi_{k},r)}\left(  V_{\varepsilon_{k}}^{\xi_{k}}\left(  \exp_{\xi_{k}}%
^{-1}(x)\right)  \right)  ^{\frac{n(p-2)}{2}}d\mu_{g}\nonumber\\
&  \leq C\int_{B(0,r/\varepsilon)}\left(  V^{\xi}\left(  z\right)  \right)
^{\frac{n(p-2)}{2}}dz\leq C \label{eq:u3}%
\end{align}
for some positive constant $C$. Combining (\ref{eq:u1}), (\ref{eq:u2}),
(\ref{eq:u3}), we get
\begin{equation}
\frac{1}{\varepsilon_{k}^{n}}\int_{M}b(x)f^{\prime}(W_{\varepsilon_{k},\xi
_{k}})u_{k}^{2}\rightarrow1\text{\qquad as }k\rightarrow+\infty.
\label{eq:lim1}%
\end{equation}

We will show that this leads to a contradiction. Since $W_{\varepsilon_{k}%
,\xi_{k}}$ is compactly supported in $B_{g}(\xi_{k},r)$, also $u_{k}$ is
compactly supported in $B_{g}(\xi_{k},r)$. Set
\[
\tilde{u}_{k}(z):=u_{k}\left(  \exp_{\xi_{k}}(\varepsilon_{k}z)\right)
\text{\qquad for }z\in B(0,r/\varepsilon).
\]
We have that
\[
\Vert\tilde{u}_{k}\Vert_{H^{1}(\mathbb{R}^{n})}\leq C\Vert u_{k}%
\Vert_{\varepsilon_{k}}\leq C,
\]
so, up to subsequence, there exists $\tilde{u}\in H^{1}(\mathbb{R}^{n})$ such
that $\tilde{u}_{k}\rightarrow\tilde{u}$ weakly in $H^{1}(\mathbb{R}^{n})$ and
strongly in $L_{\text{loc}}^{p}(\mathbb{R}^{n})$, $p\in(2,2_{n}^{\ast})$.

From (\ref{eq:ukappa}) we deduce that%
\begin{equation}
-\Delta\tilde{u}+a(\xi)\tilde{u}=b(\xi)f^{\prime}(V^{\xi})\tilde{u}.
\label{eq:utilde}%
\end{equation}
Indeed, let $\varphi\in\mathcal{C}^{\infty}(\mathbb{R}^{n})$ and let $\rho>0$
be such that the support of $\varphi$ is contained in $B(0,\rho).$ For $k$
large we set
\[
\varphi_{k}(x):=\varphi\left(  \frac{1}{\varepsilon_{k}}\exp_{\xi_{k}}%
^{-1}(x)\right)  \chi\left(  \exp_{\xi_{k}}^{-1}(x)\right)  \text{\qquad for
}x\in B_{g}(\xi_{k},\varepsilon_{k}\rho).
\]
Thus, by (\ref{eq:ukappa}), we have
\begin{multline*}
\frac{1}{\varepsilon_{k}^{n}}\int_{M}\varepsilon_{k}^{2}c(x)\nabla_{g}%
u_{k}\nabla_{g}\varphi_{k}+a(x)u_{k}\varphi_{k}d\mu_{g}=\frac{1}%
{\varepsilon_{k}^{n}}\int_{M}b(x)f^{\prime}(W_{\varepsilon_{k},\xi_{k}}%
)u_{k}\varphi_{k}d\mu_{g}\\
+\frac{1}{\varepsilon_{k}^{n}}\int_{M}b(x)f^{\prime}(W_{\varepsilon_{k}%
,\xi_{k}})(\psi_{k}+\zeta_{k})\varphi_{k}d\mu_{g},
\end{multline*}
where the last term vanishes because $\psi_{k}\rightarrow0$ and $\zeta
_{k}\rightarrow0$ in $H_{\varepsilon}^{1}$. Moreover, we have
\begin{align*}
&  \frac{1}{\varepsilon_{k}^{n}}\int_{M}\varepsilon_{k}^{2}c(x)\nabla_{g}%
u_{k}\nabla_{g}\varphi_{k}\\
&  =\int_{B(0,\rho)}%
%TCIMACRO{\tsum \limits_{l,m=1}^{n}}%
%BeginExpansion
{\textstyle\sum\limits_{l,m=1}^{n}}
%EndExpansion
g^{lm}(\varepsilon_{k}z)c(\exp_{\xi_{k}}(\varepsilon_{k}z))\frac
{\partial\left(  \tilde{u}_{k}(z)\chi(\varepsilon_{k}z)\right)  }{\partial
z_{l}}\frac{\partial\varphi}{\partial z_{m}}|g(\varepsilon_{k}z)|^{\frac{1}%
{2}}dz\\
&  =c(\xi_{k})\int_{B(0,\rho)}\nabla\tilde{u}(z)\nabla\varphi(z)dz+o(1)\\
&  =c(\xi)\int_{B(0,\rho)}\nabla\tilde{u}(z)\nabla\varphi(z)dz+o(1),
\end{align*}%
\begin{align*}
\frac{1}{\varepsilon_{k}^{n}}\int_{M}a(x)u_{k}\varphi_{k}d\mu_{g}  &
=\int_{B(0,\rho)}a(\exp_{\xi_{k}}(\varepsilon_{k}z))\tilde{u}_{k}%
(z)\varphi(z)|g(\varepsilon_{k}z)|^{\frac{1}{2}}dz\\
&  =a(\xi_{k})\int_{B(0,\rho)}\tilde{u}_{k}(z)\varphi(z)dz+o(1)\\
&  =a(\xi)\int_{B(0,\rho)}\tilde{u}_{k}(z)\varphi(z)dz+o(1)
\end{align*}
and, analogously,
\[
\frac{1}{\varepsilon_{k}^{n}}\int_{M}b(x)f^{\prime}(W_{\varepsilon_{k},\xi
_{k}})u_{k}\varphi_{k}d\mu_{g}=b(\xi)\int_{B(0,\rho)}f^{\prime}(V^{\xi}%
)\tilde{u}_{k}(z)\varphi(z)dz+o(1),
\]
so we get (\ref{eq:utilde}).

Next we prove that
\begin{equation}
\int_{\mathbb{R}_{+}^{n}}\left(  c(\xi)\nabla\psi^{h}\nabla\tilde{u}%
+a(\xi)\psi^{h}\tilde{u}\right)  dz=0\text{\qquad for all }h\in1,\dots,n-1.
\label{eq:scalare}%
\end{equation}
In fact, since $\phi_{k},\psi_{k}\in K_{\varepsilon,\xi}^{\bot}$ and
$\Vert\zeta_{k}\Vert_{\varepsilon_{k}}\rightarrow0$, we have
\begin{equation}
\left\vert \left\langle Z_{\varepsilon_{k},\xi_{k}}^{h},u_{k}\right\rangle
_{\varepsilon_{k}}\right\vert =\left\vert \left\langle Z_{\varepsilon_{k}%
,\xi_{k}}^{h},\zeta_{k}\right\rangle _{\varepsilon_{k}}\right\vert \leq\Vert
Z_{\varepsilon_{k},\xi_{k}}^{h}\Vert_{\varepsilon_{k}}\Vert\zeta_{k}%
\Vert_{\varepsilon_{k}}=o(1). \label{eq:zetau}%
\end{equation}
On the other hand, using the Taylor expansion of $g^{ij}(\varepsilon z)$,
$|g(\varepsilon z)|^{\frac{1}{2}}$ we get
\begin{align}
\left\langle Z_{\varepsilon_{k},\xi_{k}}^{h},u_{k}\right\rangle _{\varepsilon
_{k}}  &  =\frac{1}{\varepsilon_{k}^{n}}\int_{M}\varepsilon_{k}^{2}\nabla
_{g}Z_{\varepsilon_{k},\xi_{k}}^{h}\nabla_{g}u_{k}+a(x)Z_{\varepsilon_{k}%
,\xi_{k}}^{h}u_{k}\nonumber\\
&  =\int_{B(0,r/\varepsilon_{k})}\sum_{l,m=1}^{n}g^{lm}(\varepsilon
_{k}z)c(\exp_{\xi_{k}}(\varepsilon z))\frac{\partial\left(  \psi^{h}%
(z)\chi(\varepsilon_{k}z)\right)  }{\partial z_{l}}\frac{\partial\tilde{u}%
_{k}}{\partial z_{m}}|g(\varepsilon_{k}z)|^{\frac{1}{2}}dz\nonumber\\
&  \qquad+\int_{B(0,r/\varepsilon_{k})}a(\exp_{\xi_{k}}(\varepsilon
z))\psi^{h}(z)\chi(\varepsilon_{k}z)\tilde{u}_{k}|g(\varepsilon_{k}%
z)|^{\frac{1}{2}}dz\nonumber\\
&  =\int_{\mathbb{R}_{+}^{n}}\left(  c(\xi)\nabla\psi^{h}\nabla\tilde{u}%
+a(\xi)\psi^{h}\tilde{u}\right)  dz+o(1), \label{eq:zetau2}%
\end{align}
because $c(\exp_{\xi_{k}}(\varepsilon z))=c(\xi_{k})+O(\varepsilon_{k}%
)=c(\xi)+o(1)$ and $a(\exp_{\xi_{k}}(\varepsilon z))=a(\xi)+o(1)$. So, from
(\ref{eq:zetau}) and (\ref{eq:zetau2}) we obtain (\ref{eq:scalare}).

Now, (\ref{eq:scalare}) and (\ref{eq:utilde}) imply that $\tilde{u}=0$.
Therefore,
\begin{align*}
\frac{1}{\varepsilon_{k}^{n}}\int_{M}f^{\prime}(W_{\varepsilon_{k},\xi_{k}%
}(x))u_{k}^{2}(x)d\mu_{g}  &  \leq\frac{1}{\varepsilon_{k}^{n}}\int
_{B_{g}(0,r)}f^{\prime}\left(  V_{\varepsilon}^{\xi}\left(  \exp_{\xi_{k}%
}^{-1}(x)\right)  \right)  u_{k}^{2}(x)d\mu_{g}\\
&  =C\int_{B(0,r/\varepsilon_{k})}f^{\prime}(V^{\xi}(z))\tilde{u}_{k}%
^{2}(z)=o(1),
\end{align*}
which contradicts (\ref{eq:lim1}). This concludes the proof.
\end{proof}

\begin{lemma}
\label{lem:3.3}There exist $\varepsilon_{0}>0$ and $C>0$ such that
\[
\Vert R_{\varepsilon,\xi}\Vert_{\varepsilon}\leq C\varepsilon\text{ for all
}\varepsilon\in(0,\varepsilon_{0}).
\]

\end{lemma}

\begin{proof}
We recall that, if $\tilde{v}(\eta):=v(\exp_{\xi}(\eta))$, then
\begin{equation}
\Delta_{g}v=\Delta\tilde{v}+(g_{\xi}^{ij}-\delta_{ij})\partial_{ij}\tilde
{v}-g_{\xi}^{ij}\Gamma_{ij}^{k}\partial_{k}\tilde{v}, \label{eq:Laplcoord}%
\end{equation}
where $\Gamma_{ij}^{k}$ are the Christoffel symbols. Let $G_{\varepsilon,\xi}$
be a function such that $W_{\varepsilon,\xi}=i_{\varepsilon}^{\ast
}(b(x)G_{\varepsilon,\xi})$, i.e.
\[
-\varepsilon^{2}\Delta_{g}W_{\varepsilon,\xi}+A(x)W_{\varepsilon,\xi
}=B(x)G_{\varepsilon,\xi}.
\]
Then we have
\begin{align*}
B(x)G_{\varepsilon,\xi}  &  =-\varepsilon^{2}\Delta_{g}W_{\varepsilon,\xi
}+A(x)W_{\varepsilon,\xi}\\
=  &  -\varepsilon^{2}\Delta(V_{\varepsilon}^{\xi}(\eta)\chi(\eta
))-\varepsilon^{2}(g_{\xi}^{ij}-\delta_{ij})\partial_{ij}(V_{\varepsilon}%
^{\xi}(\eta)\chi(\eta))\\
&  +\varepsilon^{2}g_{\xi}^{ij}\Gamma_{ij}^{k}\partial_{k}(V_{\varepsilon
}^{\xi}(\eta)\chi(\eta))+A(\exp_{\xi}(\eta))V_{\varepsilon}^{\xi}(\eta
)\chi(\eta)\\
=  &  -\varepsilon^{2}V_{\varepsilon}^{\xi}(\eta)\Delta\chi(\eta
)-2\varepsilon^{2}\nabla V_{\varepsilon}^{\xi}(\eta)\nabla\chi(\eta
)-\varepsilon^{2}(g_{\xi}^{ij}-\delta_{ij})\partial_{ij}(V_{\varepsilon}^{\xi
}(\eta)\chi(\eta))\\
&  +\varepsilon^{2}g_{\xi}^{ij}\Gamma_{ij}^{k}\partial_{k}(V_{\varepsilon
}^{\xi}(\eta)\chi(\eta))+[A(\exp_{\xi}(\eta))-A(\xi)]V_{\varepsilon}^{\xi
}(\eta)\chi(\eta)\\
&  +B(\xi)V_{\varepsilon}^{\xi}(\eta)\chi(\eta)
\end{align*}
using (\ref{eq:Vxi}). By definition of $R_{\varepsilon,\xi}$ and inequality
(\ref{eq:istar}) we have that
\[
\Vert R_{\varepsilon,\xi}\Vert_{\varepsilon}\leq\Vert i_{\varepsilon}^{\ast
}(b(x)f(W_{\varepsilon,\xi}))-W_{\varepsilon,\xi}\Vert_{\varepsilon}%
\leq|W_{\varepsilon,\xi}^{p-1}-G_{\varepsilon,\xi}|_{p^{\prime},\varepsilon}.
\]
Now
\begin{align*}
\int_{M}|W_{\varepsilon,\xi}^{p-1}-G_{\varepsilon,\xi}|^{p^{\prime}}d\mu
_{g}\leq &  C\int_{B(0,r)}\left\vert \left(  V_{\varepsilon}^{\xi}(\eta
)\chi(\eta)\right)  ^{p-1}-G_{\varepsilon,\xi}(\exp_{\xi}(\eta))\right\vert
^{p^{\prime}}d\eta\\
\leq &  C\int_{B(0,r)}\left\vert \left(  V_{\varepsilon}^{\xi}(\eta)\right)
^{p-1}[\chi(\eta)^{p-1}-\chi(\eta)]\right\vert ^{p^{\prime}}d\eta\\
&  +C\varepsilon^{2p^{\prime}}\int_{B(0,r)}\left\vert V_{\varepsilon}^{\xi
}(\eta)\Delta\chi(\eta)\right\vert ^{p^{\prime}}d\eta\\
&  +C\varepsilon^{2p^{\prime}}\int_{B(0,r)}\left\vert \nabla V_{\varepsilon
}^{\xi}(\eta)\nabla\chi(\eta)\right\vert ^{p^{\prime}}d\eta\\
&  +C\varepsilon^{2p^{\prime}}\int_{B(0,r)}\left\vert (g_{\xi}^{ij}%
(\eta)-\delta_{ij})\partial_{ij}(V_{\varepsilon}^{\xi}(\eta)\chi
(\eta))\right\vert ^{p^{\prime}}d\eta\\
&  +C\varepsilon^{2p^{\prime}}\int_{B(0,r)}\left\vert g_{\xi}^{ij}(\eta
)\Gamma_{ij}^{k}(\eta)\partial_{k}(V_{\varepsilon}^{\xi}(\eta)\chi
(\eta))\right\vert ^{p^{\prime}}d\eta\\
&  +C\int_{B(0,r)}\left\vert [A(\exp_{\xi}(\eta))-a(\xi)]V_{\varepsilon}^{\xi
}(\eta)\chi(\eta)\right\vert ^{p^{\prime}}d\eta\\
&  =I_{1}+I_{2}+I_{3}+I_{4}+I_{5}+I_{6}.
\end{align*}
Using the exponential decay of $V_{\varepsilon}^{\xi}$ and its derivative we
get
\[
I_{1}+I_{2}+I_{3}=o(\varepsilon^{n+2p^{\prime}}).
\]
Moreover,
\begin{align*}
I_{4}  &  =C\varepsilon^{2p^{\prime}}\int_{B(0,r)}\left\vert (g_{\xi}%
^{ij}(\eta)-\delta_{ij})\partial_{ij}(V_{\varepsilon}^{\xi}(\eta)\chi
(\eta))\right\vert ^{p^{\prime}}d\eta\\
&  =C\varepsilon^{2p^{\prime}}\int_{B(0,r)}\left\vert (g_{\xi}^{ij}%
(\eta)-\delta_{ij})[\partial_{ij}V_{\varepsilon}^{\xi}(\eta)]\chi
(\eta)\right\vert ^{p^{\prime}}d\eta+o(\varepsilon^{n+2p^{\prime}})\\
&  =C\varepsilon^{n}\int_{B(0,r/\varepsilon)}\left\vert (g_{\xi}%
^{ij}(\varepsilon z)-\delta_{ij})[\partial_{ij}V^{\xi}(z)]\chi(\varepsilon
z)\right\vert ^{p^{\prime}}dz+o(\varepsilon^{n+2p^{\prime}})\\
&  =O(\varepsilon^{n+2p^{\prime}})
\end{align*}
in light of the Taylor expansion of $g_{\xi}^{ij}(\varepsilon z),$ and
\begin{align*}
I_{5}  &  =C\varepsilon^{2p^{\prime}}\int_{B(0,r)}\left\vert g_{\xi}^{ij}%
(\eta)\Gamma_{ij}^{k}(\eta)\partial_{k}(V_{\varepsilon}^{\xi}(\eta)\chi
(\eta))\right\vert ^{p^{\prime}}d\eta\\
&  =C\varepsilon^{2p^{\prime}}\int_{B(0,r)}\left\vert \Gamma_{ij}^{k}%
(\eta)[\partial_{k}V_{\varepsilon}^{\xi}(\eta)]\chi(\eta)\right\vert
^{p^{\prime}}d\eta+o(\varepsilon^{n+2p^{\prime}})\\
&  =C\varepsilon^{p^{\prime}+N}\int_{B(0,r/\varepsilon)}\left\vert \Gamma
_{ij}^{k}(\varepsilon z)[\partial_{k}V^{\xi}(z)]\chi(\varepsilon z)\right\vert
^{p^{\prime}}dz+o(\varepsilon^{n+2p^{\prime}})\\
&  =O(\varepsilon^{n+2p^{\prime}})
\end{align*}
since $\Gamma_{ij}^{k}(\varepsilon z)=\Gamma_{ij}^{k}(0)+O(\varepsilon
|z|)=O(\varepsilon|z|)$. Finally, setting $\tilde{A}(\eta)=A(\exp_{\xi}%
(\eta))$ and using the Taylor expansion of $\tilde{A}(\varepsilon z)$, we get
\begin{align*}
I_{6}  &  =C\int_{B(0,r)}\left\vert [\tilde{A}(\eta)-\tilde{A}%
(0)]V_{\varepsilon}^{\xi}(\eta)\chi(\eta)\right\vert ^{p^{\prime}}d\eta=\\
&  =C\varepsilon^{n}\int_{B(0,r/\varepsilon)}\left\vert [\tilde{A}(\varepsilon
z)-\tilde{A}(0)]V^{\xi}(z)\chi(\varepsilon z)\right\vert ^{p^{\prime}%
}dz=O(\varepsilon^{n+p^{\prime}}).
\end{align*}
This concludes the proof.
\end{proof}

\medskip

\begin{proof}
[Proof of Proposition \ref{prop:phieps}]We use a fixed point argument to show
existence of a solution to equation (\ref{eq:red1}). We define the operator
\begin{align*}
T_{\varepsilon,\xi}  &  :K_{\varepsilon,\xi}^{\bot}\rightarrow K_{\varepsilon
,\xi}^{\bot}\\
T_{\varepsilon,\xi}(\phi)  &  :=L_{\varepsilon,\xi}^{-1}\left(  N_{\varepsilon
,\xi}(\phi)+R_{\varepsilon,\xi}\right)
\end{align*}
By Lemma \ref{lem:Linv}, $T_{\varepsilon,\xi}$ is well defined and the
inequalities
\begin{align*}
\Vert T_{\varepsilon,\xi}(\phi)\Vert_{\varepsilon}  &  \leq C\left(  \Vert
N_{\varepsilon,\xi}(\phi)\Vert_{\varepsilon}+\Vert R_{\varepsilon,\xi}%
\Vert_{\varepsilon}\right) \\
\Vert T_{\varepsilon,\xi}(\phi_{1})-T_{\varepsilon,\xi}(\phi_{2}%
)\Vert_{\varepsilon}  &  \leq C\left(  \Vert N_{\varepsilon,\xi}(\phi
_{1})-N_{\varepsilon,\xi}(\phi_{2})\Vert_{\varepsilon}\right)
\end{align*}
hold true for some suitable constant $C>0$. From the mean value theorem and
inequality (\ref{eq:istar}) we get
\[
\Vert N_{\varepsilon,\xi}(\phi_{1})-N_{\varepsilon,\xi}(\phi_{2}%
)\Vert_{\varepsilon}\leq C\left\vert f^{\prime}(W_{\varepsilon,\xi}+\phi
_{2}+t(\phi_{1}-\phi_{2}))-f^{\prime}(W_{\varepsilon,\xi})\right\vert
_{\frac{p}{p-2},\varepsilon}\Vert\phi_{1}-\phi_{2}\Vert_{\varepsilon}.
\]
Using (\ref{eq:stimaf}) we can prove that $\left\vert f^{\prime}%
(W_{\varepsilon,\xi}+\phi_{2}+t(\phi_{1}-\phi_{2}))-f^{\prime}(W_{\varepsilon
,\xi})\right\vert _{\frac{p}{p-2},\varepsilon}<1$ provided $\Vert\phi_{1}%
\Vert_{\varepsilon}$ and $\Vert\phi_{2}\Vert_{\varepsilon}$ small enough.
Thus, there exists $0<C<1$ such that $\Vert T_{\varepsilon,\xi}(\phi
_{1})-T_{\varepsilon,\xi}(\phi_{2})\Vert_{\varepsilon}\leq C\Vert\phi_{1}%
-\phi_{2}\Vert_{\varepsilon}$. Moreover, from the same estimates we get
\[
\Vert N_{\varepsilon,\xi}(\phi)\Vert_{\varepsilon}\leq C\left(  \Vert\phi
\Vert_{\varepsilon}^{2}+\Vert\phi\Vert_{\varepsilon}^{p-1}\right)  .
\]
This, combined with Lemma \ref{lem:3.3}, gives
\[
\Vert T_{\varepsilon,\xi}(\phi)\Vert_{\varepsilon}\leq C\left(  \Vert
N_{\varepsilon,\xi}(\phi)\Vert_{\varepsilon}+\Vert R_{\varepsilon,\xi}%
\Vert_{\varepsilon}\right)  \leq C\left(  \Vert\phi\Vert_{\varepsilon}%
^{2}+\Vert\phi\Vert_{\varepsilon}^{p-1}+\varepsilon\right)  .
\]
So, there exists $C>0$ such that $T_{\varepsilon,\xi}$ maps the ball of center
$0$ and radius $C\varepsilon$ in $K_{\varepsilon,\xi}^{\bot}$ into itself, and
it is a contraction. It follows that $T_{\varepsilon,\xi}$ has a fixed point
$\phi_{\varepsilon,\xi}$ with norm $\Vert\phi_{\varepsilon,\xi}\Vert
_{\varepsilon}\leq\varepsilon$.

The continuity of $\phi_{\varepsilon,\xi}$ with respect to $\xi$ can be proved
by similar arguments via the implicit function theorem.
\end{proof}

\section{The reduced functional}

\label{sec:redfunc}In this section we obtain the expansion of the reduced
energy functional $\tilde{J}_{\varepsilon}(\xi):=J_{\varepsilon}%
(W_{\varepsilon,\xi}+\Phi_{\varepsilon,\xi})$ stated in Proposition
\ref{lem:tool2}.

\begin{lemma}
\label{lem:redc0}We have that
\[
\tilde{J}_{\varepsilon}(\xi):=J_{\varepsilon}(W_{\varepsilon,\xi}%
+\Phi_{\varepsilon,\xi})=J_{\varepsilon}(W_{\varepsilon,\xi})+O(\varepsilon)
\]
$\mathcal{C}^{1}$-uniformly with respect to $\xi\in M$ as $\varepsilon
\rightarrow0$.
\end{lemma}

\begin{proof}
We divide the proof into two steps: first, we show that the estimate holds
true $\mathcal{C}^{0}$-uniformly with respect to $\xi$, and then we show that
it holds true $\mathcal{C}^{1}$-uniformly as well.

\textbf{Step 1}: \ $J_{\varepsilon}(W_{\varepsilon,\xi}+\phi_{\varepsilon,\xi
})=J_{\varepsilon}(W_{\varepsilon,\xi})+O(\varepsilon)$ holds true
$\mathcal{C}^{0}$-uniformly with respect to $\xi\in M$ as $\varepsilon
\rightarrow0$.

Indeed, by (\ref{eq:red1}), we have that
\begin{align*}
J_{\varepsilon}(W_{\varepsilon,\xi}+\phi_{\varepsilon,\xi})-J_{\varepsilon
}(W_{\varepsilon,\xi})=  &  \frac{1}{2}\Vert\phi_{\varepsilon,\xi}%
\Vert_{\varepsilon}^{2}+\frac{1}{\varepsilon^{n}}\int\varepsilon^{2}c(x)\nabla
W_{\varepsilon,\xi}\nabla\phi_{\varepsilon,\xi}+a(x)W_{\varepsilon,\xi}%
\phi_{\varepsilon,\xi}\\
&  -\frac{1}{\varepsilon^{n}}\int b(x)[F(W_{\varepsilon,\xi}+\phi
_{\varepsilon,\xi})-F(W_{\varepsilon,\xi})]\\
=  &  -\frac{1}{2}\Vert\phi_{\varepsilon,\xi}\Vert_{\varepsilon}^{2}+\frac
{1}{\varepsilon^{n}}\int b(x)[f(W_{\varepsilon,\xi}+\phi_{\varepsilon,\xi
})-f(W_{\varepsilon,\xi})]\phi_{\varepsilon,\xi}\\
&  -\frac{1}{\varepsilon^{n}}\int b(x)[F(W_{\varepsilon,\xi}+\phi
_{\varepsilon,\xi})-F(W_{\varepsilon,\xi})-f(W_{\varepsilon,\xi}%
+\phi_{\varepsilon,\xi})\phi_{\varepsilon,\xi}],
\end{align*}
where $F(u)=\frac{1}{p}|u^{+}|^{p}$. Using the mean value theorem we obtain
\begin{align*}
\int b(x)[f(W_{\varepsilon,\xi}+\phi_{\varepsilon,\xi})-f(W_{\varepsilon,\xi
})]\phi_{\varepsilon,\xi}  &  =\int b(x)f^{\prime}(W_{\varepsilon,\xi}%
+t_{1}\phi_{\varepsilon,\xi})\phi_{\varepsilon,\xi}^{2}\\
\int b(x)[F(W_{\varepsilon,\xi}+\phi_{\varepsilon,\xi})-F(W_{\varepsilon,\xi
})-f(W_{\varepsilon,\xi}+\phi_{\varepsilon,\xi})\phi_{\varepsilon,\xi}]  &
=\int b(x)f^{\prime}(W_{\varepsilon,\xi}+t_{2}\phi_{\varepsilon,\xi}%
)\phi_{\varepsilon,\xi}^{2}%
\end{align*}
for some $t_{1},t_{2}\in(0,1)$. Now, since $\Vert\phi_{\varepsilon,\xi}%
\Vert_{\varepsilon}=O(\varepsilon)$,
\begin{align*}
\frac{1}{\varepsilon^{n}}\int|b(x)f^{\prime}(W_{\varepsilon,\xi}+t_{1}%
\phi_{\varepsilon,\xi})\phi_{\varepsilon,\xi}^{2}|  &  \leq\frac
{C}{\varepsilon^{n}}\int W_{\varepsilon,\xi}^{p-2}\phi_{\varepsilon,\xi}%
^{2}+\frac{C}{\varepsilon^{n}}\int\phi_{\varepsilon,\xi}^{p}\\
&  \leq|\phi_{\varepsilon,\xi}|_{2,\varepsilon}^{2}+|\phi_{\varepsilon,\xi
}|_{p,\varepsilon}^{p}\\
&  \leq\Vert\phi_{\varepsilon,\xi}\Vert_{\varepsilon}^{2}+\Vert\phi
_{\varepsilon,\xi}\Vert_{\varepsilon}^{p}=o(\varepsilon)
\end{align*}
and the claim follows.

\textbf{Step 2}: \ Setting $\xi(y)=\exp_{\xi_{0}}(y)$, we have that
\[
\left.  \frac{\partial}{\partial y_{h}}J_{\varepsilon}(W_{\varepsilon,\xi
}+\phi_{\varepsilon,\xi})\right\vert _{y=0}=\left.  \frac{\partial}{\partial
y_{h}}J_{\varepsilon}(W_{\varepsilon,\xi(y)})\right\vert _{y=0}+O(\varepsilon
)
\]
$\mathcal{C}^{0}$-uniformly with respect to $\xi\in M$ as $\varepsilon
\rightarrow0$ for all $h=1,\dots,n$.

Since the proof of this statement is lengthy, we postpone it to Appendix
\ref{sec:appendix}.
\end{proof}

Now we expand the function $\xi\mapsto J_{\varepsilon}(W_{\varepsilon,\xi})$.

\begin{lemma}
\label{lem:expJeps}The expansion
\begin{align*}
J_{\varepsilon}(W_{\varepsilon,\xi})  &  =\frac{1}{2}\int_{\mathbb{R}^{n}%
}c(\xi)|\nabla V^{\xi}(z)|^{2}+a(\xi)\left(  V^{\xi}\right)  ^{2}%
(z)dz-\frac{1}{p}\int_{\mathbb{R}^{n}}b(\xi)\left(  V^{\xi}\right)
^{p}(z)dz+O(\varepsilon)\\
&  =\left(  \frac{p-2}{2p}\int_{\mathbb{R}^{n}}U^{p}dz\right)  \frac
{c(\xi)^{\frac{n}{2}}a(\xi)^{\frac{p}{p-2}-\frac{n}{2}}}{b(\xi)^{\frac{2}%
{p-2}}}+O(\varepsilon).
\end{align*}
holds true $\mathcal{C}^{1}$-uniformly with respect to $\xi\in M$.
\end{lemma}

\begin{proof}
We perform the proof in three steps.

\textbf{Step1}: The equality
\begin{align*}
&  \frac{1}{2}\int_{\mathbb{R}^{n}}c(\xi)|\nabla V^{\xi}(z)|^{2}+a(\xi)\left(
V^{\xi}\right)  ^{2}(z)dz-\frac{1}{p}\int_{\mathbb{R}^{n}}b(\xi)\left(
V^{\xi}\right)  ^{p}(z)dz\\
&  =\left(  \frac{p-2}{2p}\int_{\mathbb{R}^{n}}U^{p}dz\right)  \frac
{c(\xi)^{\frac{n}{2}}a(\xi)^{\frac{p}{p-2}-\frac{n}{2}}}{b(\xi)^{\frac{2}%
{p-2}}}%
\end{align*}
holds true for every $\xi\in M.$

This follows by straightforward computation using (\ref{eq:rescaling}) and
(\ref{eq:U}).

\textbf{Step2}: \ The estimate
\[
{J_{\varepsilon}(W_{\varepsilon,\xi})=\frac{1}{2}\int_{\mathbb{R}^{n}}%
c(\xi)|\nabla V^{\xi}(z)|^{2}+a(\xi)\left(  V^{\xi}\right)  ^{2}(z)dz-\frac
{1}{p}\int_{\mathbb{R}^{n}}b(\xi)\left(  V^{\xi}\right)  ^{p}%
(z)dz+O(\varepsilon)}%
\]
holds true $\mathcal{C}^{0}$-uniformly with respect to $\xi\in M$ as
$\varepsilon\rightarrow0$.

Indeed, in normal coordinates we have
\begin{align*}
J_{\varepsilon}(W_{\varepsilon,\xi})=  &  \frac{1}{2}\int_{|z|<\frac
{r}{\varepsilon}}\sum_{i,j=1}^{n}g^{ij}(\varepsilon z)\tilde{c}(\varepsilon
z)\frac{\partial\left(  V^{\xi}(z)\chi_{r}(\varepsilon z)\right)  }{\partial
z_{i}}\frac{\partial\left(  V^{\xi}(z)\chi_{r}(\varepsilon z)\right)
}{\partial z_{j}}|g(\varepsilon z)|^{\frac{1}{2}}dz\\
&  +\frac{1}{2}\int_{|z|<\frac{r}{\varepsilon}}\tilde{a}(\varepsilon z)\left(
V^{\xi}(z)\chi_{r}(\varepsilon z)\right)  ^{2}|g(\varepsilon z)|^{\frac{1}{2}%
}dz\\
&  -\frac{1}{p}\int_{|z|<\frac{r}{\varepsilon}}\tilde{b}(\varepsilon z)\left(
V^{\xi}(z)\chi_{r}(\varepsilon z)\right)  ^{p}|g(\varepsilon z)|^{\frac{1}{2}%
}dz,
\end{align*}
where $\varepsilon z:=\exp_{\xi}^{-1}x,$ and $\tilde{c}(\varepsilon z):=c(x)$,
$\tilde{a}(\varepsilon z):=a(x)$ and $\tilde{b}(\varepsilon z):=b(x)$. Using
the expansions (\ref{eq:espg1}) and (\ref{eq:espg2}) and collecting terms of
the same order, we get
\[
J_{\varepsilon}(W_{\varepsilon,\xi})=\frac{1}{2}\int_{\mathbb{R}^{n}}%
c(\xi)|\nabla V^{\xi}(z)|^{2}+a(\xi)\left(  V^{\xi}\right)  ^{2}(z)dz-\frac
{1}{p}\int_{\mathbb{R}^{n}}b(\xi)\left(  V^{\xi}\right)  ^{p}%
(z)dz+O(\varepsilon),
\]
as claimed.

\textbf{Step 3:} \ The estimate
\begin{multline*}
\left.  \frac{\partial}{\partial y_{h}}J_{\varepsilon}(W_{\varepsilon,\xi
(y)})\right\vert _{y=0}=\\
=\left.  \frac{\partial}{\partial y_{h}}\left(  \frac{1}{2}\int_{\mathbb{R}%
^{n}}c(\xi(y))|\nabla V^{\xi(y)}|^{2}+a(\xi(y))\left(  V^{\xi(y)}\right)
^{2}dz-\frac{1}{p}\int_{\mathbb{R}^{n}}b(\xi(y))\left(  V^{\xi(y)}\right)
^{p}(z)dz\right)  \right\vert _{y=0}+O(\varepsilon)
\end{multline*}
holds true $\mathcal{C}^{0}$-uniformly with respect to $\xi\in M$ as
$\varepsilon\rightarrow0$ for all $h=1,\dots,n$. Here $\xi(y):=\exp_{\xi_{0}%
}(y)$ with $y\in B(0,r)$.

The proof of this statement is technical and we postpone it to Appendix
\ref{sec:appendix}
\end{proof}

\appendix

\section{Some technical facts}

\label{sec:appendix}Here we collect some technical facts we have used to prove
some of the results, and we give the missing proofs.

The proofs of Lemmas \ref{lem:6.1}, \ref{lem:6.2} and \ref{lem:6.3} are
similar to those of Lemmas 6.1, 6.2 and 6.3 of \cite{mp} and will just be sketched.

\begin{lemma}
\label{lem:6.1}The following estimates hold true
\[
\left\Vert \frac{\partial}{\partial y_{h}}Z_{\varepsilon,\xi(y)}%
^{l}\right\Vert _{\varepsilon}=O\left(  \frac{1}{\varepsilon}\right)
,\qquad\left\Vert \frac{\partial}{\partial y_{h}}W_{\varepsilon,\xi
(y)}\right\Vert _{\varepsilon}=O\left(  \frac{1}{\varepsilon}\right)  .
\]

\end{lemma}

\begin{proof}
The proof follows by direct computation.
\end{proof}

\begin{lemma}
\label{lem:6.2}The following estimates hold true
\[
\left\langle Z_{\varepsilon,\xi_{0}}^{l},\left.  \frac{\partial}{\partial
y_{h}}W_{\varepsilon,\xi(y)}\right\vert _{y=0}\right\rangle _{\varepsilon
}=-\frac{1}{\varepsilon}C\delta_{hl}+o\left(  \frac{1}{\varepsilon}\right)  ,
\]
where $C\in\mathbb{R}$ is a suitable constant.
\end{lemma}

\begin{proof}
The proof follows from the definitions of $Z_{\varepsilon,\xi_{0}}^{l}$ and
$W_{\varepsilon,\xi(y)}$ and the Taylor expansion of $\frac{\partial
\mathcal{E}_{k}}{\partial y_{h}}$.
\end{proof}

\begin{lemma}
\label{lem:5.2}There exists $C>0$ such that, for $\varepsilon$ small enough,
\[
\Vert Z_{\varepsilon,\xi}^{h}-i_{\varepsilon}^{\ast}[b(x)f^{\prime
}(W_{\varepsilon,\xi(y)})Z_{\varepsilon,\xi}^{h}\Vert_{\varepsilon}\leq
C\varepsilon
\]
for all $\xi\in M,$ $h=1,\dots,n$.
\end{lemma}

\begin{proof}
Let $G_{\varepsilon,\xi}$ be a function such that $Z_{\varepsilon,\xi}%
^{h}(x)=i_{\varepsilon}^{\ast}(b(x)G_{\varepsilon,\xi})$, i.e.
\[
-\varepsilon^{2}\Delta_{g}Z_{\varepsilon,\xi}^{h}+A(x)Z_{\varepsilon,\xi}%
^{h}=B(x)G_{\varepsilon,\xi}.
\]
Thus, using (\ref{eq:Laplcoord}) we have
\begin{align*}
b(x)G_{\varepsilon,\xi}  &  =-\varepsilon^{2}\Delta_{g}Z_{\varepsilon,\xi
}+A(x)Z_{\varepsilon,\xi}\\
=  &  -\varepsilon^{2}\psi_{\xi}^{h}(\eta/\varepsilon)\Delta\chi
(\eta)-2\varepsilon^{2}\nabla\psi_{\xi}^{h}(\eta/\varepsilon)\nabla\chi
(\eta)-\varepsilon^{2}(g_{\xi}^{ij}-\delta_{ij})\partial_{ij}(\psi_{\xi}%
^{h}(\eta/\varepsilon)\chi(\eta))\\
&  +\varepsilon^{2}g_{\xi}^{ij}\Gamma_{ij}^{k}\partial_{k}(\psi_{\xi}^{h}%
(\eta/\varepsilon)\chi(\eta))+[A(\exp_{\xi}(\eta))-a(\xi)]\psi_{\xi}^{h}%
(\eta/\varepsilon)\chi(\eta)\\
&  +(p-1)B(\xi)\left(  V_{\varepsilon}^{\xi}(\eta)\right)  ^{p-2}\psi_{\xi
}^{h}(\eta/\varepsilon)\chi(\eta),
\end{align*}
by (\ref{eq:lin}). Now the proof follows as in Lemma \ref{lem:3.3}.
\end{proof}

\begin{lemma}
\label{lem:6.3}There exists $C>0$ such that, for $\varepsilon$ small enough,
\[
\left\Vert \frac{\partial}{\partial y_{h}}W_{\varepsilon,\xi(y)}+\frac
{1}{\varepsilon}Z_{\varepsilon,\xi(y)}^{h}\right\Vert _{\varepsilon}\leq
C\varepsilon
\]
for all $\xi_{0}\in M,$ $h=1,\dots,n$.
\end{lemma}

\begin{proof}
The proof follows from the definitons of $Z_{\varepsilon,\xi_{0}}^{l}$ and
$W_{\varepsilon,\xi(y)}$ and the Taylor expansion of $\frac{\partial
\mathcal{E}_{k}}{\partial y_{h}}$.
\end{proof}

\begin{proof}
[Proof of Lemma \ref{lem:redc0}, Step 2.]We have
\begin{align*}
\left.  \frac{\partial}{\partial y_{h}}J_{\varepsilon}(W_{\varepsilon,\xi
(y)}+\phi_{\varepsilon,\xi(y)})-J_{\varepsilon}(W_{\varepsilon,\xi
(y)})\right\vert _{y=0} &  =J_{\varepsilon}^{\prime}(W_{\varepsilon,\xi}%
+\phi_{\varepsilon,\xi})\left[  \left.  \frac{\partial}{\partial y_{h}}\left(
W_{\varepsilon,\xi(y)}+\phi_{\varepsilon,\xi(y)}\right)  \right\vert
_{y=0}\right]  \\
&  \qquad-J_{\varepsilon}^{\prime}(W_{\varepsilon,\xi})\left[  \left.
\frac{\partial}{\partial y_{h}}W_{\varepsilon,\xi(y)}\right\vert
_{y=0}\right]  \\
&  =\left(  J_{\varepsilon}^{\prime}(W_{\varepsilon,\xi}+\phi_{\varepsilon
,\xi})-J_{\varepsilon}^{\prime}(W_{\varepsilon,\xi})\right)  \left[  \left.
\frac{\partial}{\partial y_{h}}W_{\varepsilon,\xi(y)}\right\vert
_{y=0}\right]  \\
&  \qquad+J_{\varepsilon}^{\prime}(W_{\varepsilon,\xi}+\phi_{\varepsilon,\xi
})\left[  \left.  \frac{\partial}{\partial y_{h}}\phi_{\varepsilon,\xi
(y)}\right\vert _{y=0}\right]  .
\end{align*}
In light of (\ref{eq:red1}), the last term is
\[
J_{\varepsilon}^{\prime}(W_{\varepsilon,\xi}+\phi_{\varepsilon,\xi})\left[
\left.  \frac{\partial}{\partial y_{h}}\phi_{\varepsilon,\xi(y)}\right\vert
_{y=0}\right]  =\sum_{l=i}^{n}C_{\varepsilon}^{l}\left\langle Z_{\varepsilon
,\xi}^{l},\left.  \frac{\partial}{\partial y_{h}}\phi_{\varepsilon,\xi
(y)}\right\vert _{y=0}\right\rangle _{\varepsilon}.
\]
Since $\phi_{\varepsilon,\xi}\in K_{\varepsilon,\xi}^{\bot}$, we have
\[
\left\langle Z_{\varepsilon,\xi(y)}^{l},\frac{\partial}{\partial y_{h}}%
\phi_{\varepsilon,\xi(y)}\right\rangle _{\varepsilon}=-\left\langle
\frac{\partial}{\partial y_{h}}Z_{\varepsilon,\xi(y)}^{l},\phi_{\varepsilon
,\xi(y)}\right\rangle _{\varepsilon}=O\left(  \left\Vert \frac{\partial
}{\partial y_{h}}Z_{\varepsilon,\xi(y)}^{l}\right\Vert _{\varepsilon}\cdot
\phi_{\varepsilon,\xi(y)}\right)  =O(1)
\]
by Lemma \ref{lem:6.1}. Moreover, we claim that
\begin{equation}
\sum_{l=1}^{n}|C_{\varepsilon}^{l}|=O(\varepsilon).\label{eq:Ceps}%
\end{equation}
Indeed, since $\phi_{\varepsilon,\xi}\in K_{\varepsilon,\xi}^{\bot},$
\begin{align*}
J_{\varepsilon}^{\prime}(W_{\varepsilon,\xi}+\phi_{\varepsilon,\xi
})[Z_{\varepsilon,\xi}^{l}] &  =\left\langle W_{\varepsilon,\xi}%
,Z_{\varepsilon,\xi}^{l}\right\rangle _{\varepsilon}-\frac{1}{\varepsilon^{n}%
}\int b(x)f(W_{\varepsilon,\xi})Z_{\varepsilon,\xi}^{l}\\
&  +\frac{1}{\varepsilon^{n}}\int b(x)\left(  f(W_{\varepsilon,\xi}%
+\phi_{\varepsilon,\xi})-f(W_{\varepsilon,\xi})\right)  Z_{\varepsilon,\xi
}^{l}.
\end{align*}
Now, from the exponantial decay of $V^{\xi}$ and its derivatives, and the
expantion of $g^{ij}$ and $|g|,$ we get
\begin{multline*}
\left\langle W_{\varepsilon,\xi},Z_{\varepsilon,\xi}^{l}\right\rangle
_{\varepsilon}-\frac{1}{\varepsilon^{n}}\int b(x)f(W_{\varepsilon,\xi
})Z_{\varepsilon,\xi}^{l}\\
=\frac{1}{2}\int_{\mathbb{R}^{n}}c(x)\nabla V^{\xi}\nabla\psi_{\xi}%
^{l}+a(x)V^{\xi}\psi_{\xi}^{l}-\frac{1}{p}\int_{\mathbb{R}^{n}}b(x)(V^{\xi
})^{p-1}\psi_{\xi}^{l}+O(\varepsilon^{2})=O(\varepsilon^{2})
\end{multline*}
and, by the mean value theorem, for some $t\in(0,1)$ we have that
\begin{multline*}
\frac{1}{\varepsilon^{n}}\left\vert \int\left(  f(W_{\varepsilon,\xi}%
+\phi_{\varepsilon,\xi})-f(W_{\varepsilon,\xi})\right)  Z_{\varepsilon,\xi
}^{l}\right\vert =\frac{1}{\varepsilon^{n}}\left\vert \int f^{\prime
}(W_{\varepsilon,\xi}+t\phi_{\varepsilon,\xi})\phi_{\varepsilon,\xi
}Z_{\varepsilon,\xi}^{l}\right\vert \\
\leq C|W_{\varepsilon,\xi}|_{p,\varepsilon}^{p-2}|\phi_{\varepsilon,\xi
}|_{p,\varepsilon}|Z_{\varepsilon,\xi}^{l}|_{p,\varepsilon}+C|\phi
_{\varepsilon,\xi}|_{p,\varepsilon}^{p-1}|Z_{\varepsilon,\xi}^{l}%
|_{p,\varepsilon}=O(\varepsilon),
\end{multline*}
Thus, (\ref{eq:Ceps}) holds true and
\[
J_{\varepsilon}^{\prime}(W_{\varepsilon,\xi}+\phi_{\varepsilon,\xi})\left[
\left.  \frac{\partial}{\partial y_{h}}\phi_{\varepsilon,\xi(y)}\right\vert
_{y=0}\right]  =O(\varepsilon).
\]
We prove now that
\[
\left(  J_{\varepsilon}^{\prime}(W_{\varepsilon,\xi}+\phi_{\varepsilon,\xi
})-J_{\varepsilon}^{\prime}(W_{\varepsilon,\xi})\right)  \left[  \left.
\frac{\partial}{\partial y_{h}}W_{\varepsilon,\xi(y)}\right\vert
_{y=0}\right]  =O(\varepsilon).
\]
We have%
\begin{align*}
& \left(  J_{\varepsilon}^{\prime}(W_{\varepsilon,\xi(y)}+\phi_{\varepsilon
,\xi(y)})-J_{\varepsilon}^{\prime}(W_{\varepsilon,\xi(y)})\right)  \left[
\frac{\partial}{\partial y_{h}}W_{\varepsilon,\xi(y)}\right]  \\
& =\left\langle \phi_{\varepsilon,\xi(y)},\frac{\partial}{\partial y_{h}%
}W_{\varepsilon,\xi(y)}\right\rangle _{\varepsilon}-\frac{1}{\varepsilon^{n}%
}\int b(x)(f(W_{\varepsilon,\xi(y)}+\phi_{\varepsilon,\xi(y)}%
)-f(W_{\varepsilon,\xi(y)}))\frac{\partial}{\partial y_{h}}W_{\varepsilon
,\xi(y)}\\
& =\left\langle \phi_{\varepsilon,\xi(y)}-i_{\varepsilon}^{\ast}%
[b(x)f^{\prime}(W_{\varepsilon,\xi(y)})\phi_{\varepsilon,\xi(y)}%
],\frac{\partial}{\partial y_{h}}W_{\varepsilon,\xi(y)}\right\rangle
_{\varepsilon}\\
& -\frac{1}{\varepsilon^{n}}\int b(x)(f(W_{\varepsilon,\xi(y)}+\phi
_{\varepsilon,\xi(y)})-f(W_{\varepsilon,\xi(y)})-f^{\prime}(W_{\varepsilon
,\xi(y)})\phi_{\varepsilon,\xi(y)})\frac{\partial}{\partial y_{h}%
}W_{\varepsilon,\xi(y)}\\
& =\left\langle \phi_{\varepsilon,\xi(y)}-i_{\varepsilon}^{\ast}%
[b(x)f^{\prime}(W_{\varepsilon,\xi(y)})\phi_{\varepsilon,\xi(y)}%
],\frac{\partial}{\partial y_{h}}W_{\varepsilon,\xi(y)}+\frac{1}{\varepsilon
}Z_{\varepsilon,\xi(y)}^{h}\right\rangle _{\varepsilon}\\
& -\frac{1}{\varepsilon}\left\langle \phi_{\varepsilon,\xi(y)},Z_{\varepsilon
,\xi(y)}^{h}-i_{\varepsilon}^{\ast}[b(x)f^{\prime}(W_{\varepsilon,\xi
(y)})Z_{\varepsilon,\xi(y)}^{h}]\right\rangle _{\varepsilon}\\
& -\frac{1}{\varepsilon^{n}}\int b(x)(f(W_{\varepsilon,\xi(y)}+\phi
_{\varepsilon,\xi(y)})-f(W_{\varepsilon,\xi(y)})-f^{\prime}(W_{\varepsilon
,\xi(y)})\phi_{\varepsilon,\xi(y)})\frac{\partial}{\partial y_{h}%
}W_{\varepsilon,\xi(y)}\\
& :=I_{1}+I_{2}+I_{3}.
\end{align*}
By Lemma \ref{lem:5.2} we have that
\[
\left\vert I_{2}\right\vert \leq\frac{1}{\varepsilon}\Vert Z_{\varepsilon,\xi
}^{h}-i_{\varepsilon}^{\ast}[b(x)f^{\prime}(W_{\varepsilon,\xi(y)}%
)Z_{\varepsilon,\xi}^{h}\Vert_{\varepsilon}\Vert\phi_{\varepsilon,\xi(y)}%
\Vert_{\varepsilon}=O(\varepsilon).
\]
In order to estimate $I_{3}$ we observe that
\begin{equation}
|f^{\prime}(W_{\varepsilon,\xi}+v)-f^{\prime}(W_{\varepsilon,\xi}%
)|\leq\left\{
\begin{array}
[c]{ll}%
CW_{\varepsilon,\xi}^{p-3}|v| & 2<p<3,\\
C(W_{\varepsilon,\xi}^{p-3}|v|+|v|^{p-2}) & p\geq3.
\end{array}
\right.  \label{eq:stimaf}%
\end{equation}
We prove that $I_{3}=O(\varepsilon)$ only for $p\geq3$; the other case is
similar. By mean value theorem, (\ref{eq:stimaf}) and Lemma \ref{lem:6.1} we
have that
\begin{align*}
\left\vert I_{3}\right\vert  &  \leq\frac{C}{\varepsilon^{n}}\int\left(
W_{\varepsilon,\xi(y)}^{p-3}\phi_{\varepsilon,\xi(y)}^{2}+|\phi_{\varepsilon
,\xi(y)}|^{p-1}\right)  \left\vert \frac{\partial}{\partial y_{h}%
}W_{\varepsilon,\xi(y)}\right\vert \\
&  \leq\left(  |W_{\varepsilon,\xi(y)}|_{p,\varepsilon}^{p-3}|\phi
_{\varepsilon,\xi(y)}|_{2,\varepsilon}^{2}+|\phi_{\varepsilon,\xi
(y)}|_{p,\varepsilon}^{p-1}\right)  \left\vert \frac{\partial}{\partial y_{h}%
}W_{\varepsilon,\xi(y)}\right\vert _{p,\varepsilon}\\
&  \leq\left(  \Vert\phi_{\varepsilon,\xi(y)}\Vert_{\varepsilon}^{2}+\Vert
\phi_{\varepsilon,\xi(y)}\Vert_{\varepsilon}^{p-1}\right)  \left\Vert
\frac{\partial}{\partial y_{h}}W_{\varepsilon,\xi(y)}\right\Vert
_{\varepsilon}=O(\varepsilon)
\end{align*}
Finally, for $I_{1}$ we have that
\[
\left\vert I_{1}\right\vert \leq\Vert\phi_{\varepsilon,\xi(y)}-i_{\varepsilon
}^{\ast}[b(x)f^{\prime}(W_{\varepsilon,\xi(y)})\phi_{\varepsilon,\xi(y)}%
]\Vert_{\varepsilon}\left\Vert \frac{\partial}{\partial y_{h}}W_{\varepsilon
,\xi(y)}+\frac{1}{\varepsilon}Z_{\varepsilon,\xi(y)}^{h}\right\Vert
_{\varepsilon}\leq O(\varepsilon),
\]
because
\begin{align*}
\Vert\phi_{\varepsilon,\xi(y)}-i_{\varepsilon}^{\ast}[b(x)f^{\prime
}(W_{\varepsilon,\xi(y)})\phi_{\varepsilon,\xi(y)}]\Vert_{\varepsilon} &
\leq\Vert\phi_{\varepsilon,\xi(y)}\Vert_{\varepsilon}+\Vert i_{\varepsilon
}^{\ast}[b(x)f^{\prime}(W_{\varepsilon,\xi(y)})\phi_{\varepsilon,\xi(y)}%
]\Vert_{\varepsilon}\\
&  \leq\Vert\phi_{\varepsilon,\xi(y)}\Vert_{\varepsilon}+C|f^{\prime
}(W_{\varepsilon,\xi(y)})\phi_{\varepsilon,\xi(y)}|_{p^{\prime},\varepsilon}\\
&  \leq\Vert\phi_{\varepsilon,\xi(y)}\Vert_{\varepsilon}+C|W_{\varepsilon
,\xi(y)}|_{p,\varepsilon}^{p-2}|\phi_{\varepsilon,\xi(y)}|_{p,\varepsilon
}=O(\varepsilon)
\end{align*}
and, by Lemma \ref{lem:6.3},
\[
\left\Vert \frac{\partial}{\partial y_{h}}W_{\varepsilon,\xi(y)}+\frac
{1}{\varepsilon}Z_{\varepsilon,\xi(y)}^{h}\right\Vert _{\varepsilon}\leq
C\varepsilon.
\]
This concludes the proof.
\end{proof}

\begin{proof}
[Proof of Lemma \ref{lem:expJeps}, Step 3.]We prove the claim for $h=1$.
By(\ref{eq:derWeps}) we have
\begin{align*}
\left.  \frac{\partial}{\partial y_{1}}J_{\varepsilon}(W_{\varepsilon,\xi
(y)})\right\vert _{y=0} &  =\left.  J_{\varepsilon}^{\prime}(W_{\varepsilon
,\xi(y)})\left[  \frac{\partial}{\partial y_{1}}W_{\varepsilon,\xi(y)}\right]
\right\vert _{y=0}\\
&  =\left.  \frac{\varepsilon^{2}}{\varepsilon^{n}}\int_{M}c(x)\nabla
_{g}W_{\varepsilon,\xi(y)}\nabla_{g}\frac{\partial}{\partial y_{1}%
}W_{\varepsilon,\xi(y)}d\mu_{g}\right\vert _{y=0}\\
&  +\left.  \frac{1}{\varepsilon^{n}}\int_{M}a(x)W_{\varepsilon,\xi(y)}%
\frac{\partial}{\partial y_{1}}W_{\varepsilon,\xi(y)}d\mu_{g}\right\vert
_{y=0}\\
&  +\left.  \frac{1}{\varepsilon^{n}}\int_{M}b(x)W_{\varepsilon,\xi(y)}%
^{p-1}\frac{\partial}{\partial y_{1}}W_{\varepsilon,\xi(y)}d\mu_{g}\right\vert
_{y=0}\\
&  =:I_{1}+I_{2}+I_{3}.
\end{align*}
To simplify the notation, set $x:=\exp_{\xi_{0}}(\eta)$ and
\[
\tilde{\mathcal{E}}(y,\eta):=\exp_{\xi(y)}^{-1}\exp_{\xi_{0}}(\eta)=\exp
_{\xi(y)}^{-1}(x)=\mathcal{E}(y,x),
\]
so
\[
\tilde{\mathcal{E}}(0,\varepsilon z)=\exp_{\xi_{0}}^{-1}\exp_{\xi_{0}%
}(\varepsilon z)=\varepsilon z.
\]
For the term $I_{1}$we have
\begin{align*}
I_{1} &  =\left.  \frac{\varepsilon^{2}}{\varepsilon^{n}}\int_{M}%
c(x)\nabla_{g}W_{\varepsilon,\xi(y)}\nabla_{g}\frac{\partial}{\partial y_{1}%
}W_{\varepsilon,\xi(y)}d\mu_{g}\right\vert _{y=0}\\
&  =\frac{\varepsilon^{2}}{\varepsilon^{n}}\int_{\mathbb{R}^{n}}c(\exp
_{\xi_{0}}\eta)|g_{\xi_{0}}(\eta)|^{\frac{1}{2}}g_{\xi_{0}}^{ij}(\eta
)\frac{\partial}{\partial\eta_{i}}\left[  \tilde{\gamma}(0)U_{\varepsilon
}(\sqrt{\tilde{A}(0)}\tilde{\mathcal{E}}(0,\eta))\chi(\tilde{\mathcal{E}%
}(0,\eta))\right]  \\
&  \times\left.  \frac{\partial}{\partial\eta_{j}}\frac{\partial}{\partial
y_{1}}\left[  \tilde{\gamma}(y)U_{\varepsilon}(\sqrt{\tilde{A}(y)}%
\tilde{\mathcal{E}}(y,\eta))\chi(\tilde{\mathcal{E}}(y,\eta))\right]
\right\vert _{y=0}d\eta\\
&  =\int_{\mathbb{R}^{n}}\tilde{c}(\varepsilon z)|g_{\xi_{0}}(\varepsilon
z)|^{\frac{1}{2}}g_{\xi_{0}}^{ij}(\varepsilon z)\frac{\partial}{\partial
z_{i}}\left[  \tilde{\gamma}(0)U_{\varepsilon}(\sqrt{\tilde{A}(0)}%
\tilde{\mathcal{E}}(0,\varepsilon z))\chi(\tilde{\mathcal{E}}(0,\varepsilon
z))\right]  \\
&  \times\left.  \frac{\partial}{\partial z_{j}}\frac{\partial}{\partial
y_{1}}\left[  \tilde{\gamma}(y)U_{\varepsilon}(\sqrt{\tilde{A}(y)}%
\tilde{\mathcal{E}}(y,\varepsilon z))\chi(\tilde{\mathcal{E}}(y,\varepsilon
z))\right]  \right\vert _{y=0}dz.
\end{align*}
Now, recalling that $\tilde{\mathcal{E}}(0,\varepsilon z)=\exp_{\xi_{0}}%
^{-1}\exp_{\xi_{0}}(\varepsilon z)=\varepsilon z$, from equations
(\ref{eq:derWeps}), (\ref{eq:espE}), (\ref{eq:espg1}), (\ref{eq:espg2}) and
the exponential decay of $U$ and its derivatives, we get
\begin{align*}
I_{1} &  =\int_{\mathbb{R}^{n}}\tilde{c}(\varepsilon z)\tilde{\gamma
}(0)\left[  \left(  \frac{\partial}{\partial z_{i}}U(\sqrt{\tilde{A}%
(0)}z)\right)  \chi(\varepsilon z)+U(\sqrt{\tilde{A}(0)}z)\frac{\partial
}{\partial z_{i}}\chi(\varepsilon z)\right]  \\
&  \times\left.  \frac{\partial}{\partial z_{i}}\frac{\partial}{\partial
y_{1}}\left[  \tilde{\gamma}(y)U_{\varepsilon}(\sqrt{\tilde{A}(y)}%
\tilde{\mathcal{E}}(y,\varepsilon z))\chi(\tilde{\mathcal{E}}(y,\varepsilon
z))\right]  \right\vert _{y=0}d\eta+O(\varepsilon)\\
&  =\int_{\mathbb{R}^{n}}\tilde{c}(\varepsilon z)\tilde{\gamma}(0)\frac
{\partial}{\partial z_{i}}\left(  U(\sqrt{\tilde{A}(0)}z)\right)  \{\left.
\frac{\partial}{\partial y_{1}}\tilde{\gamma}(y)\right\vert _{y=0}%
\frac{\partial}{\partial z_{i}}U(\sqrt{\tilde{A}(0)}z)\\
&  +\tilde{\gamma}(0)\frac{\partial}{\partial z_{i}}U(\sqrt{\tilde{A}%
(0)}z)\left.  \frac{\partial}{\partial y_{1}}\chi(\tilde{\mathcal{E}%
}(y,\varepsilon z))\right\vert _{y=0}\\
&  +\tilde{\gamma}(0)\left.  \frac{\partial}{\partial y_{1}}\frac{\partial
}{\partial z_{i}}U_{\varepsilon}(\sqrt{\tilde{A}(y)}\tilde{\mathcal{E}%
}(y,\varepsilon z))\right\vert _{y=0}\}dz+O(\varepsilon)=:D_{1}+D_{2}%
+D_{3}+O(\varepsilon).
\end{align*}
Expanding
\[
\tilde{c}(\varepsilon z)=\tilde{c}(0)+\varepsilon\frac{\partial\tilde{c}%
}{\partial z_{k}}(0)z_{k}+O(\varepsilon^{2}|z|)
\]
we get
\[
D_{1}=\frac{1}{2}\left.  \frac{\partial}{\partial y_{1}}(\tilde{\gamma
}(y))^{2}\right\vert _{y=0}\tilde{c}(0)\int_{\mathbb{R}^{n}}|\nabla_{z}\left(
U(\sqrt{\tilde{A}(0)}z)\right)  |^{2}dz+O(\varepsilon)
\]
and, by (\ref{eq:espE}),
\begin{align*}
D_{2} &  =\tilde{\gamma}^{2}(0)\int_{\mathbb{R}^{n}}\tilde{c}(\varepsilon
z)|\nabla_{z}U(\sqrt{\tilde{A}(0)}z)|^{2}\chi^{\prime}(\varepsilon
|z|)\frac{z_{k}}{|z|}\left.  \frac{\partial}{\partial y_{1}}\tilde
{\mathcal{E}}_{k}(y,\varepsilon z))\right\vert _{y=0}dz=\\
&  =-\tilde{\gamma}^{2}(0)\tilde{c}(0)\int_{\mathbb{R}^{n}}|\nabla_{z}%
U(\sqrt{\tilde{A}(0)}z)|^{2}\chi^{\prime}(\varepsilon|z|)\frac{z_{1}}%
{|z|}dz+O(\varepsilon)=O(\varepsilon),
\end{align*}
since the last integral is zero by symmetry reasons. At this point we observe
that
\begin{align*}
\frac{\partial U}{\partial z_{1}}(\sqrt{\tilde{A}(0)}z) &  =U^{\prime}%
(\sqrt{\tilde{A}(0)}z)\frac{z_{1}}{|z|},\\
\frac{\partial}{\partial z_{1}}\left(  \frac{U^{\prime}(\sqrt{\tilde{A}(0)}%
z)}{|z|}\right)   &  =\left(  \frac{\sqrt{\tilde{A}(0)}U^{\prime\prime}%
(z)}{|z|^{2}}-\frac{U^{\prime}(z)}{|z|^{3}}\right)  z_{1},
\end{align*}
where, abusing notation, we write
\[
\frac{\partial U}{\partial z_{1}}(\sqrt{\tilde{A}(0)}z)=\left.  \frac
{\partial}{\partial\eta_{1}}U(\eta)\right\vert _{\eta=\sqrt{\tilde{A}(0)}z}.
\]
In the same spirit, in the following we write
\[
\frac{\partial^{2}U}{\partial z_{i}\partial z_{j}}(\sqrt{\tilde{A}%
(0)}z)=\left.  \frac{\partial^{2}}{\partial\eta_{i}\partial\eta_{j}}%
U(\eta)\right\vert _{\eta=\sqrt{\tilde{A}(0)}z}.
\]
We have%
\begin{align*}
& \frac{D_{3}}{\tilde{\gamma}^{2}(0)}=\int_{\mathbb{R}^{n}}\tilde
{c}(\varepsilon z)\frac{\partial}{\partial z_{i}}\left(  U(\sqrt{\tilde{A}%
(0)}z)\right)  \left.  \frac{\partial}{\partial z_{i}}\frac{\partial}{\partial
y_{1}}U_{\varepsilon}(\sqrt{\tilde{A}(y)}\tilde{\mathcal{E}}(y,\varepsilon
z))\right\vert _{y=0}dz\\
& =\int_{\mathbb{R}^{n}}\tilde{c}(\varepsilon z)\frac{\partial}{\partial
z_{i}}\left(  U(\sqrt{\tilde{A}(0)}z)\right)  \frac{\partial}{\partial z_{i}%
}\left\{  \frac{\partial U}{\partial z_{k}}(\sqrt{\tilde{A}(0)}z)\left(
\left.  \frac{\partial\sqrt{\tilde{A}(y)}}{\partial y_{1}}\right\vert
_{y=0}z_{k}\left.  +\frac{\sqrt{\tilde{A}(0)}}{\varepsilon}\frac{\partial
}{\partial y_{1}}\tilde{\mathcal{E}}(y,\varepsilon z)\right\vert
_{y=0}\right)  \right\}  dz\\
& =\int_{\mathbb{R}^{n}}\tilde{c}(\varepsilon z)\left.  \frac{\partial
\sqrt{\tilde{A}(y)}}{\partial y_{1}}\right\vert _{y=0}\frac{\partial}{\partial
z_{i}}\left(  U(\sqrt{\tilde{A}(0)}z)\right)  \frac{\partial}{\partial z_{i}%
}\left(  \frac{\partial U}{\partial z_{k}}(\sqrt{\tilde{A}(0)}z)z_{k}\right)
dz\\
& -\frac{1}{\varepsilon}\int_{\mathbb{R}^{n}}\tilde{c}(\varepsilon z)\tilde
{A}(0)\frac{\partial U}{\partial z_{i}}(\sqrt{\tilde{A}(0)}z)\frac{\partial
}{\partial z_{i}}\left(  \frac{\partial U}{\partial z_{k}}(\sqrt{\tilde{A}%
(0)}z)\delta_{1k}\right)  dz+O(\varepsilon)\\
& =\int_{\mathbb{R}^{n}}\tilde{c}(\varepsilon z)\left.  \frac{\partial
\sqrt{\tilde{A}(y)}}{\partial y_{1}}\right\vert _{y=0}\frac{\partial}{\partial
z_{i}}\left(  U(\sqrt{\tilde{A}(0)}z)\right)  \frac{\partial U}{\partial
z_{i}}(\sqrt{\tilde{A}(0)}z)dz\\
& +\int_{\mathbb{R}^{n}}\tilde{c}(\varepsilon z)\left.  \frac{\partial
\sqrt{\tilde{A}(y)}}{\partial y_{1}}\right\vert _{y=0}\sqrt{\tilde{A}(0)}%
\frac{\partial}{\partial z_{i}}\left(  U(\sqrt{\tilde{A}(0)}z)\right)
\frac{\partial^{2}U}{\partial z_{i}z_{k}}(\sqrt{\tilde{A}(0)}z)z_{k}dz\\
& -\frac{1}{\varepsilon}\int_{\mathbb{R}^{n}}\tilde{c}(\varepsilon z)\tilde
{A}(0)\frac{\partial U}{\partial z_{i}}(\sqrt{\tilde{A}(0)}z)\frac{\partial
}{\partial z_{i}}\left(  \frac{\partial U}{\partial z_{1}}(\sqrt{\tilde{A}%
(0)}z)\right)  dz+O(\varepsilon)\\
& =\frac{1}{2}\int_{\mathbb{R}^{n}}\tilde{c}(0)\left.  \frac{\partial
}{\partial y_{1}}\left\vert \nabla_{z}U(\sqrt{\tilde{A}(y)}z)\right\vert
^{2}\right\vert _{y=0}dz\\
& -\frac{1}{\varepsilon}\int_{\mathbb{R}^{n}}\tilde{c}(\varepsilon z)\tilde
{A}(0)\frac{\partial U}{\partial z_{i}}(\sqrt{\tilde{A}(0)}z)\frac{\partial
}{\partial z_{i}}\left(  \frac{\partial U}{\partial z_{1}}(\sqrt{\tilde{A}%
(0)}z)\right)  dz+O(\varepsilon).
\end{align*}
Moreover, expanding $\tilde{c}$ around $z=0,$%
\begin{align*}
& -\frac{1}{\varepsilon}\int_{\mathbb{R}^{n}}\tilde{c}(\varepsilon z)\tilde
{A}(0)\frac{\partial U}{\partial z_{i}}(\sqrt{\tilde{A}(0)}z)\frac{\partial
}{\partial z_{i}}\left(  \frac{\partial U}{\partial z_{1}}(\sqrt{\tilde{A}%
(0)}z)\right)  dz\\
& =-\frac{1}{\varepsilon}\int_{\mathbb{R}^{n}}\tilde{c}(0)\tilde{A}%
(0)\frac{\partial U}{\partial z_{i}}(\sqrt{\tilde{A}(0)}z)\frac{\partial
}{\partial z_{i}}\left(  \frac{\partial U}{\partial z_{1}}(\sqrt{\tilde{A}%
(0)}z)\right)  dz\\
& -\int_{\mathbb{R}^{n}}\frac{\partial\tilde{c}}{\partial z_{k}}(0)z_{k}%
\tilde{A}(0)\frac{\partial U}{\partial z_{i}}(\sqrt{\tilde{A}(0)}%
z)\frac{\partial}{\partial z_{i}}\left(  \frac{\partial U}{\partial z_{1}%
}(\sqrt{\tilde{A}(0)}z)\right)  dz+O(\varepsilon)\\
& =-\frac{1}{\varepsilon}\tilde{c}(0)\tilde{A}(0)\int_{\mathbb{R}^{n}%
}U^{\prime}(\sqrt{\tilde{A}(0)}z)\frac{z_{i}}{|z|}\frac{\partial}{\partial
z_{i}}\left(  \frac{U^{\prime}(\sqrt{\tilde{A}(0)}z)}{|z|}z_{1}\right)  \\
& -\int_{\mathbb{R}^{n}}\frac{\partial\tilde{c}}{\partial z_{k}}(0)z_{k}%
\tilde{A}(0)U^{\prime}(\sqrt{\tilde{A}(0)}z)\frac{z_{i}}{|z|}\frac{\partial
}{\partial z_{i}}\left(  \frac{U^{\prime}(\sqrt{\tilde{A}(0)}z)}{|z|}%
z_{1}\right)  +O(\varepsilon)\\
& =-\int_{\mathbb{R}^{n}}\frac{\partial\tilde{c}}{\partial z_{k}}(0)\tilde
{A}(0)\left(  \frac{U^{\prime}(\sqrt{\tilde{A}(0)}z)}{|z|}\right)  ^{2}%
z_{k}z_{1}dz\\
& -\int_{\mathbb{R}^{n}}\frac{\partial\tilde{c}}{\partial z_{k}}(0)\tilde
{A}(0)\frac{U^{\prime}(\sqrt{\tilde{A}(0)}z)}{|z|}\left(  \frac{\sqrt
{\tilde{A}(0)}U^{\prime\prime}(z)}{|z|^{2}}-\frac{U^{\prime}(z)}{|z|^{3}%
}\right)  |z|^{2}z_{k}z_{1}dz+O(\varepsilon)\\
& =-\int_{\mathbb{R}^{n}}\frac{\partial\tilde{c}}{\partial z_{1}}(0)\tilde
{A}(0)\left(  \frac{U^{\prime}(\sqrt{\tilde{A}(0)}z)}{|z|}\right)  ^{2}%
z_{1}^{2}dz\\
& -\int_{\mathbb{R}^{n}}\frac{\partial\tilde{c}}{\partial z_{1}}(0)\tilde
{A}(0)\frac{U^{\prime}(\sqrt{\tilde{A}(0)}z)}{|z|}\left(  \frac{\sqrt
{\tilde{A}(0)}U^{\prime\prime}(z)}{|z|^{2}}-\frac{U^{\prime}(z)}{|z|^{3}%
}\right)  |z|^{2}z_{1}^{2}dz+O(\varepsilon)\\
& =-\frac{1}{2}\frac{\partial\tilde{c}}{\partial z_{1}}(0)\int_{\mathbb{R}%
^{n}}\frac{\partial}{\partial z_{1}}\left(  \left\vert \nabla_{z}%
U(\sqrt{\tilde{A}(y)}z)\right\vert ^{2}\right)  z_{1}dz+O(\varepsilon)\\
& =\frac{1}{2}\int_{\mathbb{R}^{n}}\left.  \frac{\partial\tilde{c}}{\partial
y_{1}}(y)\right\vert _{y=0}\left\vert \nabla_{z}U(\sqrt{\tilde{A}%
(y)}z)\right\vert ^{2}dz+O(\varepsilon),
\end{align*}
since the other integrals are zero by symmetry reasons. In conclusion, we
have
\[
I_{1}=\left.  \frac{\partial}{\partial y_{1}}\left(  \frac{1}{2}%
\int_{\mathbb{R}^{n}}c(\xi(y))\left\vert \nabla_{z}V^{\xi(y)}(z)\right\vert
^{2}dz\right)  \right\vert _{y=0}+O(\varepsilon)
\]
For the second term, by (\ref{eq:espg1}) and (\ref{eq:espg2}) we have, in an
analogous way,
\begin{align}
I_{2} &  =\left.  \frac{1}{\varepsilon^{n}}\int_{M}a(x)W_{\varepsilon,\xi
(y)}\frac{\partial}{\partial y_{1}}W_{\varepsilon,\xi(y)}d\mu_{g}\right\vert
_{y=0}\label{eq:I2}\\
&  =\frac{1}{\varepsilon^{n}}\int_{\mathbb{R}^{n}}|g_{\xi_{0}}(\eta
)|^{\frac{1}{2}}a(\exp_{\xi_{0}}\eta)\tilde{\gamma}(0)U_{\varepsilon}%
(\sqrt{\tilde{A}(0)}\tilde{\mathcal{E}}(0,\eta))\chi(\tilde{\mathcal{E}%
}(0,\eta))\nonumber\\
&  \times\left.  \frac{\partial}{\partial y_{1}}\left[  \tilde{\gamma
}(y)U_{\varepsilon}(\sqrt{\tilde{A}(y)}\tilde{\mathcal{E}}(y,\eta))\chi
(\tilde{\mathcal{E}}(y,\eta))\right]  \right\vert _{y=0}d\eta\nonumber\\
&  =\int_{\mathbb{R}^{n}}|g_{\xi_{0}}(\varepsilon z)|^{\frac{1}{2}}\tilde
{a}(\varepsilon z)\tilde{\gamma}(0)U_{\varepsilon}(\sqrt{\tilde{A}(0)}%
\tilde{\mathcal{E}}(0,\varepsilon z))\chi(\tilde{\mathcal{E}}(0,\varepsilon
z))\nonumber\\
&  \times\left.  \frac{\partial}{\partial y_{1}}\left[  \tilde{\gamma
}(y)U_{\varepsilon}(\sqrt{\tilde{A}(y)}\tilde{\mathcal{E}}(y,\varepsilon
z))\chi(\tilde{\mathcal{E}}(y,\varepsilon z))\right]  \right\vert
_{y=0}dz\nonumber\\
&  =\int_{\mathbb{R}^{n}}\tilde{a}(\varepsilon z)\tilde{\gamma}(0)U(\sqrt
{\tilde{A}(0)}z)\nonumber\\
&  \times\left.  \frac{\partial}{\partial y_{1}}\left[  \tilde{\gamma
}(y)U_{\varepsilon}(\sqrt{\tilde{A}(y)}\tilde{\mathcal{E}}(y,\varepsilon
z))\chi(\tilde{\mathcal{E}}(y,\varepsilon z))\right]  \right\vert
_{y=0}dz+O(\varepsilon^{2})=\nonumber\\
&  =\int_{\mathbb{R}^{n}}\tilde{a}(\varepsilon z)\tilde{\gamma}(0)U^{2}%
(\sqrt{\tilde{A}(0)}z)\left.  \frac{\partial}{\partial y_{1}}\tilde{\gamma
}(y)\right\vert _{y=0}dz\nonumber\\
&  +\int_{\mathbb{R}^{n}}\tilde{a}(\varepsilon z)\tilde{\gamma}^{2}%
(0)U(\sqrt{\tilde{A}(0)}z)\left.  \frac{\partial}{\partial y_{1}}\left[
U_{\varepsilon}(\sqrt{\tilde{A}(y)}\tilde{\mathcal{E}}(y,\varepsilon
z))\right]  \right\vert _{y=0}dz+O(\varepsilon^{2}),\nonumber
\end{align}
and for the last term we have%
\begin{align}
& \int_{\mathbb{R}^{n}}\tilde{a}(\varepsilon z)\tilde{\gamma}^{2}%
(0)U(\sqrt{\tilde{A}(0)}z)\left.  \frac{\partial}{\partial y_{1}}\left[
U_{\varepsilon}(\sqrt{\tilde{A}(y)}\tilde{\mathcal{E}}(y,\varepsilon
z))\right]  \right\vert _{y=0}dz\label{eq:I2bis}\\
& =\tilde{\gamma}^{2}(0)\left.  \frac{\partial}{\partial y_{1}}\sqrt{\tilde
{A}(y)}\right\vert _{y=0}\int_{\mathbb{R}^{n}}\tilde{a}(\varepsilon
z)U(\sqrt{\tilde{A}(0)}z)\frac{\partial U}{\partial z_{k}}(\sqrt{\tilde{A}%
(0)}z)z_{k}dz\nonumber\\
& +\frac{\tilde{\gamma}^{2}(0)\sqrt{\tilde{A}(0)}}{\varepsilon}\int
_{\mathbb{R}^{n}}\tilde{a}(\varepsilon z)U(\sqrt{\tilde{A}(0)},z)\frac
{\partial U}{\partial z_{k}}(\sqrt{\tilde{A}(0)}z)\left.  \frac{\partial
}{\partial y_{1}}\tilde{\mathcal{E}_{k}}(y,\varepsilon z)\right\vert
_{y=0}dz\nonumber\\
& =\tilde{\gamma}^{2}(0)\tilde{a}(0)\left.  \frac{\partial}{\partial y_{1}%
}\sqrt{\tilde{A}(y)}\right\vert _{y=0}\int_{\mathbb{R}^{n}}U(\sqrt{\tilde
{A}(0)}z)\frac{\partial U}{\partial z_{k}}(\sqrt{\tilde{A}(0)}z)z_{k}%
dz\nonumber\\
& -\frac{\tilde{\gamma}^{2}(0)\sqrt{\tilde{A}(0)}}{\varepsilon}\int
_{\mathbb{R}^{n}}\tilde{a}(\varepsilon z)U(\sqrt{\tilde{A}(0)},z)\frac
{\partial U}{\partial z_{1}}(\sqrt{\tilde{A}(0)}z)dz+O(\varepsilon).\nonumber
\end{align}
At this point we observe that $\frac{\partial U}{\partial z_{1}}(\sqrt
{\tilde{A}(0)}z)=U^{\prime}(\sqrt{\tilde{A}(0)}z)\frac{z_{1}}{|z|}$ and,
expanding
\[
\tilde{a}(\varepsilon z)=\tilde{a}(0)+\varepsilon\frac{\partial\tilde{a}%
}{\partial z_{k}}(0)z_{k}+O(\varepsilon^{2}|z|),
\]
we obtain%
\begin{align}
& \frac{1}{\varepsilon}\int_{\mathbb{R}^{n}}\tilde{a}(\varepsilon
z)U(\sqrt{\tilde{A}(0)},z)U^{\prime}(\sqrt{\tilde{A}(0)}z)\frac{z_{1}}%
{|z|}dz\label{eq:I2ter}\\
& =\frac{1}{\varepsilon}\int_{\mathbb{R}^{n}}\tilde{a}(0)U(\sqrt{\tilde{A}%
(0)},z)U^{\prime}(\sqrt{\tilde{A}(0)}z)\frac{z_{1}}{|z|}dz\nonumber\\
& +\int_{\mathbb{R}^{n}}\frac{\partial\tilde{a}}{\partial z_{k}}%
(0)U(\sqrt{\tilde{A}(0)},z)U^{\prime}(\sqrt{\tilde{A}(0)}z)\frac{z_{1}z_{k}%
}{|z|}dz+O(\varepsilon)\nonumber\\
& =\int_{\mathbb{R}^{n}}\frac{\partial\tilde{a}}{\partial z_{1}}%
(0)U(\sqrt{\tilde{A}(0)},z)U^{\prime}(\sqrt{\tilde{A}(0)}z)\frac{z_{1}^{2}%
}{|z|}dz+O(\varepsilon).\nonumber
\end{align}
In conclusion, from (\ref{eq:I2}), (\ref{eq:I2bis}) and (\ref{eq:I2ter}) we
get
\begin{align*}
I_{2} &  =\tilde{a}(0)\tilde{\gamma}(0)\left.  \frac{\partial}{\partial y_{1}%
}\tilde{\gamma}(y)\right\vert _{y=0}\int_{\mathbb{R}^{n}}U^{2}(\sqrt{\tilde
{A}(0)}z)dz\\
&  +\tilde{\gamma}^{2}(0)\tilde{a}(0)\left.  \frac{\partial}{\partial y_{1}%
}\sqrt{\tilde{A}(y)}\right\vert _{y=0}\sum_{k}\int_{\mathbb{R}^{n}}%
U(\sqrt{\tilde{A}(0)}z)U^{\prime}(\sqrt{\tilde{A}(0)}z)\frac{z_{k}^{2}}%
{|z|}dz\\
&  -\tilde{\gamma}^{2}(0)\sqrt{\tilde{A}(0)}\frac{\partial\tilde{a}}{\partial
z_{1}}(0)\int_{\mathbb{R}^{n}}U(\sqrt{\tilde{A}(0)},z)U^{\prime}(\sqrt
{\tilde{A}(0)}z)\frac{z_{1}^{2}}{|z|}dz+O(\varepsilon)\\
&  =\left.  \frac{\partial}{\partial y_{1}}\left(  \frac{1}{2}\int
_{\mathbb{R}^{n}}a(\xi(y))\left(  V^{\xi(y)}\right)  ^{2}dz\right)
\right\vert _{y=0}+O(\varepsilon)
\end{align*}
We argue in a similar way for $I_{3}$ to complete the proof.
\end{proof}


\bigskip

\begin{thebibliography}{99}                                                                                               %


\bibitem {acp}\textsc{N. Ackermann, M. Clapp, A. Pistoia,} Boundary clustered
layers for some supercritical problems. \emph{J. Differential Equations}
\textbf{254} (2013), 4168-4183.

\bibitem {am}\textsc{A. Ambrosetti, A. Malchiodi,} Perturbation methods and
semilinear elliptic problems on $\mathbb{R}^{n}$. Progress in Mathematics
\textbf{240}. \emph{Birkh\"{a}user Verlag, Basel,} 2006.

\bibitem {amn1}\textsc{A. Ambrosetti, A. Malchiodi, W.-M. Ni,} Singularly
perturbed elliptic equations with symmetry: existence of solutions
concentrating on spheres. I. \emph{Comm. Math. Phys.} \textbf{235} (2003), 427--466.

\bibitem {amn2}\textsc{A. Ambrosetti, A. Malchiodi, W.-M. Ni,} Singularly
perturbed elliptic equations with symmetry: existence of solutions
concentrating on spheres. II. \emph{Indiana Univ. Math. J.} \textbf{53}
(2004), 297--329.

\bibitem {bd}\textsc{M. Badiale, T. D'Aprile,} Concentration around a sphere
for a singularly perturbed Sch\"{o}dinger equation. \emph{Nonlinear Anal. TMA}
\textbf{49} (2002), 947--985.

\bibitem {bw}\textsc{P. Baird, J.C. Wood,} Harmonic morphisms between
Riemannian manifolds. London Mathematical Society Monographs. New Series
\textbf{29}. \emph{The Clarendon Press, Oxford University Press, Oxford,} 2003.

\bibitem {bbm}\textsc{V. Benci, C. Bonanno, A.M. Micheletti,} On the
multiplicity of solutions of a nonlinear elliptic problem on Riemannian
manifolds. \emph{J. Funct. Anal.} \textbf{252} (2007), no. 2, 464--489.

\bibitem {be-d}\textsc{V. Benci, T. D'Aprile,} The semiclassical limit of the
nonlinear Schr\"{o}dinger equation in a radial potential. \emph{J.
Differential Equations} \textbf{184} (2002), 109--138.

\bibitem {bp}\textsc{J. Byeon, J. Park,} Singularly perturbed nonlinear
elliptic problems on manifolds. \emph{Calc. Var. Partial Differential
Equations} \textbf{24} (2005), no. 4, 459--477.

\bibitem {cy}\textsc{W. Chen, J. Yang,} Multiple solutions for nonlinear
elliptic equations on Riemannian manifolds. \emph{Electron. J. Differential
Equations} \textbf{2009}, No. 131, 16 pp.

\bibitem {cfp}\textsc{M. Clapp, J. Faya, A. Pistoia,} Nonexistence and
multiplicity of solutions to elliptic problems with supercritical exponents.
\emph{Calc. Var. Partial Differential Equations}, DOI: 10.1007/s00526-012-0564-6.

\bibitem {cfp2}\textsc{M. Clapp, J. Faya, A. Pistoia,} Positive solutions to a
supercritical elliptic problem\ which concentrate along a thin spherical hole.
Preprint arXiv:1304.1907.

\bibitem {dmp}\textsc{E.N. Dancer, A.M. Micheletti, A. Pistoia,} Multipeak
solutions for some singularly perturbed nonlinear elliptic problems on
Riemannian manifolds. \emph{Manuscripta Math.} \textbf{128} (2009), 163--193.

\bibitem {dl}\textsc{F. Dobarro, E. Lami Dozo,} Scalar curvature and warped
products of Riemann manifolds. \emph{Trans. Amer. Math. Soc.} \textbf{303}
(1987), 161--168.

\bibitem {er}\textsc{J. Eells, A. Ratto,} Harmonic maps and minimal immersions
with symmetries. Methods of ordinary differential equations applied to
elliptic variational problems. Annals of Mathematics Studies \textbf{130}.
\emph{Princeton University Press, Princeton, NJ,} 1993.

\bibitem {f}\textsc{B. Fuglede,} Harmonic morphisms between {R}iemannian
manifolds. \emph{Ann. Inst. Fourier Grenoble} \textbf{28} (1978), 107-144.

\bibitem {gm}\textsc{M. Ghimenti, A.M. Micheletti,} On the number of nodal
solutions for a nonlinear elliptic problem on symmetric Riemannian manifolds.
\emph{Proceedings of the 2007 Conference on Variational and Topological
Methods: Theory, Applications, Numerical Simulations, and Open Problems,}
15--22, Electron. J. Differ. Equ. Conf. \textbf{18}, \emph{Southwest Texas
State Univ., San Marcos TX,} 2010.

\bibitem {h}\textsc{N. Hirano,} Multiple existence of solutions for a
nonlinear elliptic problem on a Riemannian manifold. \emph{Nonlinear Anal.}
\textbf{70} (2009), 671--692.

\bibitem {kp}\textsc{S. Kim, A. Pistoia,} Clustered boundary layer sign
changing solutions for a supercritical problem. \emph{J. London Math. Soc.},
to appear.

\bibitem {kp2}\textsc{S. Kim, A. Pistoia,} Boundary towers of layers for some
supercritical problems. Preprint arXiv:1302.1217.

\bibitem {mm}\textsc{A. Malchiodi, M. Montenegro,} Boundary concentration
phenomena for a singularly perturbed elliptic problem. \emph{Comm. Pure Appl.
Math.} \textbf{55} (2002), 1507--1568.

\bibitem {mp}\textsc{A.M. Micheletti, A. Pistoia,} The role of the scalar
curvature in a nonlinear elliptic problem on Riemannian manifolds. \emph{Calc.
Var. Partial Differential Equations} \textbf{34} (2009), 233--265.

\bibitem {mp2}\textsc{A.M. Micheletti, A. Pistoia,} Nodal solutions for a
singularly perturbed nonlinear elliptic problem on Riemannian manifolds.
\emph{Adv. Nonlinear Stud.} \textbf{9} (2009), 565--577.

\bibitem {pp}\textsc{F. Pacella, A. Pistoia,} Bubble concentration on spheres
for supercritical elliptic problems. Preprint.

\bibitem {ps}\textsc{F. Pacella, P.N. Srikanth,} A reduction method for
semilinear elliptic equations and solutions concentrating on spheres. Preprint arXiv:1210.0782.

\bibitem {rs}\textsc{B. Ruf, P.N. Srikanth,} Singularly perturbed elliptic
equations with solutions concentrating on a 1-dimensional orbit, \emph{J. Eur.
Math. Soc. (JEMS)} \textbf{12} (2010), 413--427.

\bibitem {RSp}\textsc{B. Ruf, P.N. Srikanth,} Concentration on Hopf fibres for
singularly perturbed elliptic equations, preprint.

\bibitem {w}\textsc{J.C. Wood,} Harmonic morphisms between Riemannian
manifolds. \emph{Modern trends in geometry and topology,} 397--414, \emph{Cluj
Univ. Press, Cluj-Napoca,} 2006.

\bibitem {v}\textsc{D. Visetti,} Multiplicity of solutions of a zero mass
nonlinear equation on a Riemannian manifold. \emph{J. Differential Equations}
\textbf{245} (2008), 2397--2439.

\bibitem {wy}\textsc{J. Wei, S. Yan,} Infinitely many positive solutions for
an elliptic problem with critical or supercritical growth. \emph{J. Math.
Pures Appl.} \textbf{96} (2011), 307--333.
\end{thebibliography}
\end{document}